\documentclass[11pt]{article}
\usepackage{amssymb}
\usepackage{amsmath}
\usepackage{amsthm}
\usepackage[textwidth=16.5cm, textheight=25cm]{geometry}

\usepackage[english]{babel}
\addto\captionsenglish{}
\usepackage[utf8]{inputenc}
\usepackage[T1]{fontenc}
\usepackage{todonotes}
\usepackage{xcolor}

\usepackage[pdfborder={0 0 0}]{hyperref}
\usepackage{setspace}
\usepackage{tikz-cd}
\usepackage{adjustbox}
\usepackage[shortlabels]{enumitem}
\usetikzlibrary{calc}
\hypersetup{
  pdfinfo={
    Author={Joachim Jelisiejew},
    Title={Elementary components of Hilbert schemes of~points},
    Keywords={Hilbert scheme of points, Murphy's law, components, small
    tangent space},
  }
}
\numberwithin{figure}{section}
\makeatletter
\let\c@equation\c@figure
\makeatother
\numberwithin{equation}{section}
\newtheorem{theorem}{Theorem}[section]
\newtheorem{corollary}[theorem]{Corollary}
\newtheorem{lemma}[theorem]{Lemma}
\newtheorem{question}[theorem]{Question}
\newtheorem{conjecture}[theorem]{Conjecture}

\newtheorem{proposition}[theorem]{Proposition}
\theoremstyle{definition}

\newtheorem{remark}[theorem]{Remark}
\newtheorem{example}[theorem]{Example}
\renewcommand{\leq}{\leqslant}
\renewcommand{\geq}{\geqslant}

\DeclareMathOperator{\im}{im}
\DeclareMathOperator{\Spec}{Spec}

\DeclareMathOperator{\Supp}{Supp}
\DeclareMathOperator{\Tor}{Tor}
\DeclareMathOperator{\Ext}{Ext}
\DeclareMathOperator{\Hom}{Hom}

\DeclareMathOperator{\OHilb}{Hilb}

\DeclareMathOperator{\Gr}{Gr}
\DeclareMathOperator{\gr}{gr}
\DeclareMathOperator{\charr}{char}

\DeclareMathOperator{\Sym}{Sym}
\DeclareMathOperator{\Set}{Set}
\DeclareMathOperator{\Sch}{Sch}
\DeclareMathOperator{\id}{id}
\newcommand{\tensor}{\otimes}
\newcommand{\onto}{\twoheadrightarrow}
\newcommand{\into}{\hookrightarrow}

\newcommand{\pts}{\mathrm{pt}}%

\newcommand{\Hlong}[2]{\OHilb_{#1}(#2)}

\begin{document}

\newcommand{\kk}{\Bbbk}%
\newcommand{\kkbar}{\overline{\kk}}%
\newcommand{\OO}{\mathcal{O}}%
\newcommand{\II}{\mathcal{I}}%
\newcommand{\JJ}{\mathcal{J}}%
\newcommand{\DR}{R}%
\newcommand{\DM}{M}%
\newcommand{\OR}{\OO_R}%
\newcommand{\OM}{\OO_M}%
\newcommand{\IR}{I_R}%
\newcommand{\IRhom}{I_R^{h}}%
\newcommand{\ORhom}{\OO_R^{h}}%
\newcommand{\IRhomzero}{I_R^{\mathrm{hom}}}%
\newcommand{\ORhomzero}{\OO_R^{\mathrm{hom}}}%
\newcommand{\IM}{I_M}%
\newcommand{\Ires}{J}%
\newcommand{\DS}{S}%
\newcommand{\uppleftarrow}{\psi}%
\newcommand{\upparrow}{\partial}%
\newcommand{\righttarrow}{\phi}%
\newcommand{\compo}{\mathcal{Z}}%
\newcommand{\Gmult}{\mathbb{G}_{m}}%
\newcommand{\Gadd}{\mathbb{G}_{a}}%
\newcommand{\Gaddn}{\mathbb{G}_{a}^n}%
\newcommand{\Gmultbar}{\overline{\mathbb{G}}_{m}}%
\newcommand{\ambient}{\mathbb{A}}%
\newcommand{\OB}{\OO_{B}}%
\newcommand{\thetazero}{\theta_0}%
\newcommand{\famil}{\mathcal{Z}}%

\newcommand{\BBname}{Bia{\l{}}ynicki-Birula}%
\newcommand{\SchOp}{\Sch^{\mathrm{op}}}%
\newcommand{\spann}[1]{\operatorname{span}\left( #1 \right)}%

\title{\vspace{-1.5\baselineskip}Elementary components of
Hilbert schemes of~points\vspace{-0.3\baselineskip}}
\author{Joachim Jelisiejew\thanks{Institute of Mathematics, Polish Academy of
    Sciences and Institute of Mathematics,  University of Warsaw. Email:
    \url{jjelisiejew@impan.pl}. Partially supported by NCN grant 
2013/10/E/ST1/00688.}}
\date{}
\maketitle

\vspace*{-1.8\baselineskip}
\begin{abstract}
    Consider the Hilbert scheme of points on a higher-dimensional affine
    space. Its component is \emph{elementary} if it parameterizes irreducible
    subschemes. We characterize reduced elementary components in terms of
    tangent spaces and provide a computationally efficient way of finding
    such components. As an example, we find an infinite family of elementary and
    generically smooth components on the affine four-space.
    We analyse singularities and formulate a conjecture which would imply
    the non-reducedness of the Hilbert scheme.
    Our main tool is a generalization of the \BBname{} decomposition for this
    singular scheme.
\end{abstract}
{\small \textbf{MSC classes:} 14C05, 14L30, 13D10}

\vspace*{-\baselineskip}
\section{Introduction}

    While the Hilbert scheme of points on a smooth connected surface is smooth
    and irreducible~\cite{fogarty},
    little is known about the irreducible components of Hilbert schemes of points on
    higher-dimensional varieties~\cite{aimpl}, despite much recent interest in their
    geometry.  Following~\cite{iarrobino__number_of_sings}, a component is
    \emph{elementary} if it parameterizes subschemes supported at a single
    point. All components are generically \'etale-locally products of
    elementary ones.

    Up to now, the only known method of finding elementary components is to
    \emph{construct} a locus $\mathcal{L}$ inside the Hilbert scheme, and
    \emph{verify} that the tangent space to the Hilbert scheme at a point of
    $\mathcal{L}$ has dimension $\dim \mathcal{L}$, to conclude that
    $\mathcal{L}$ contains an open neighbourhood of this point.
    Roughly speaking, the construction step is conceptual and the
    verification step is algorithmic.
    Iarrobino~\cite{iarrobino__number_of_sings, iaCompressed,
    iarrobino_kanev_book_Gorenstein_algebras}
    obtained~loci $\mathcal{L}$ by choosing general forms of prescribed degrees and taking
    their apolar algebra.  The obtained algebras are called \emph{compressed}.
    Recently, Huibregtse~\cite{HuibregtseElementary} extended Iarrobino's
    method in certain cases, taking into account the automorphisms of the
    ambient variety.
    All known elementary components come from the constructions of these two
    authors. Verification
    of tangent space dimension was done in~\cite{emsalem_iarrobino_small_tangent_space} for an explicit degree $8$
    scheme, then extended \cite{Shafarevich_Deformations_of_1de} for
    algebras of multiplicity two, and afterwards conducted for several other
    explicit cases \cite{iarrobino_kanev_book_Gorenstein_algebras,
    erman_velasco_syzygetic_smoothability, HuibregtseElementary}, see
    Remark~\ref{ref:previouswork:rmk}.  It is conjectured that the tangent
    space has correct dimension in greater number of compressed cases,
    see~\cite[\S2.3]{emsalem_iarrobino_small_tangent_space}. The main
    limitation of this approach is that one needs to construct the
    locus $\mathcal{L}$. As a side effect, all loci $\mathcal{L}$ obtained so far
    are isomorphic to open subsets of products of Grassmannians.

    \newcommand{\HshortN}{\Hlong{\pts}{\mathbb{A}^n}}%
    \newcommand{\HplusN}{\OHilb_{\pts}^{+}(\mathbb{A}^n)}%
    \newcommand{\HshortNGmult}{\OHilb_{\pts}^{\Gmult}(\mathbb{A}^n)}%
    \newcommand{\HshortNAmb}{\Hlong{\pts}{\ambient}}%
    \newcommand{\HshortNAmbGmult}{\OHilb_{\pts}^{\Gmult}(\ambient)}%
    \newcommand{\HplusNAmb}{\OHilb_{\pts}^{+}(\ambient)}%
    The aim of the present paper is to avoid the construction step entirely. We answer the following question:
    \begin{question}\label{question:givenpointOnElementary}
        How to check that a given point $[\DR]\in \HshortN$ lies on an elementary
        component?
    \end{question}

    The $n$-dimensional additive group acts on $\mathbb{A}^n$ and $\HshortN$
    by translations. Let $\DR \subset \mathbb{A}^n$ be a finite subscheme
    supported at the origin.
    The tangent map of the $\mathbb{A}^n$-orbit of $[\DR]$ is
    \[
        \spann{\partial_1, \ldots ,\partial_n}\to \Hom(\IR, \OR)_{<0}.
    \]
    We say that $\DR$ has \emph{trivial negative tangents} if this tangent map is surjective or, equivalently, $T^1(\DR)_{<0} =
    0$. If the characteristic is zero or $\IR$ is homogeneous, then $\DR$ has
    trivial negative tangents if and only if $\dim \Hom(\IR, \OR)_{<0} = n$.
    Our main result is that having trivial negative tangents is intimately connected to lying on an elementary
    component.
    \goodbreak
    \begin{theorem}[Theorems~\ref{ref:thetaetale:thm} and~\ref{ref:generictrivialtangents:thm}]\label{ref:TNTrelevance:intro}
        Let $\DR \subset \mathbb{A}^n$ be a finite subscheme supported at the origin.
        \begin{enumerate}[(1)]
            \item\label{it:TNTimpliesElementary} If $\DR \subset \mathbb{A}^n$ has trivial negative tangents, then every
                component of $\HshortN$ containing $[\DR]$ is elementary. The
                subscheme $\DR$ is not
                smoothable or cleavable unless $\DR$ is a point.
            \item\label{it:ElementaryImpliesTNT} Conversely, let $\compo
                \subset \HshortN$ be an elementary irreducible
                component. When the characteristic is zero and $\compo$ is generically reduced,
                a general point of $\compo$ has trivial negative tangents.
        \end{enumerate}
    \end{theorem}
    Theorem~\ref{ref:TNTrelevance:intro} gives an answer to
    Question~\ref{question:givenpointOnElementary} in characteristic zero, modulo generically non-reduced
    components. No examples of such components are known. Moreover, the
    answer depends only on the information contained in the tangent
    space, which a priori would seem insufficient. For instance, in a similar context
    Erman and Velasco~\cite[p.1144]{erman_velasco_syzygetic_smoothability}
    state that the ``tangent space dimension is rather coarse invariant in the
    study of smoothability''. If a component
    $\compo$ has a smooth point, then it is automatically generically
    reduced.
    Using Theorem~\ref{ref:TNTrelevance:intro}\ref{it:TNTimpliesElementary},
    we can certify that a given point $[\DR]$ lies on an elementary component
    without any knowledge of this component. This is an advantage over the
    above construct-and-verify method. This theorem is well
    suited for the search for new elementary components: we can search
    for $\DR$ having trivial negative tangents, a far more tractable property.

    Better yet, deformation theory provides sufficient conditions for smoothness
    of a point $[\DR]$ of $\HshortN$.
    Classically, the point $[\DR]\in\HshortN$ has an
    obstruction space given by the Schlessinger's functor
    $T^2(\DR) \subset \Ext^1(\IR, \OR)$. If additionally the finite subscheme $\DR$ has
    trivial negative tangents,
    we can restrict to the non-negative part of $T^2(\DR)$. We obtain the following result.
    \begin{corollary}[Corollary~\ref{ref:elementary:cor}]\label{ref:elementaryintro:cor}
        Let $\DR \subset \mathbb{A}^n$ be a finite subscheme
        supported at the origin. If $\DR$
        has trivial negative tangents and $T^2(\DR)_{\geq 0} = 0$ then
        $[\DR]$ is a smooth point of $\HshortN$ lying on a unique
        elementary component.
    \end{corollary}
    In a number of cases, the vanishing of
    $\Ext^1(\IR, \OR)_{\geq0}$ is forced by the degrees in the Betti table;
    see~Example~\ref{ex:naiveexample}.
    This makes Corollary~\ref{ref:elementaryintro:cor} effective: computer
    algebra experiments show that it is a rich source of smooth points on
    elementary components.  Of
    course, the point $[\DR]$ may be
    smooth even if $T^2(\DR)_{\geq 0}\neq 0$. A subtler relative smoothness
    criterion is given in Corollary~\ref{ref:smoothnessHilb:cor}.
    Applying it, we obtain the following infinite family of
    elementary components of $\Hlong{\pts}{\mathbb{A}^4}$.
    Let $\DS = \kk[x_1, x_2, y_1, y_2]$ and, for all $e\in \mathbb{Z}_+$, take
    a finite subscheme $\DM(e) \subset \mathbb{A}^4$ defined by ideal $(x_1,
    x_2)^e + (y_1, y_2)^e$.
    A form $s$ of degree $2(e-1)$ in $\DS/I_{\DM(e)}$ is uniquely written as
    \[
        s = \sum_{i,j} c_{ij} x_{1}^i x_2^{e-1-i}
        y_1^{j}y_2^{e-1-j},\ \mbox{where}\ c_{ij}\in \kk.
    \]
     We say that $s$ is \emph{general} if the $e \times e$ matrix $[c_{ij}]$
     is invertible. Fix a general form $s$ and let $\DR(e) \subset \DM(e)$ be
     the subscheme cut out by this form. The degree of $\DR(e)$ is $d :=
     \binom{e+1}{2}^2-1$. The subscheme $\DR(e)$ deforms freely in $\DM(e)$
     and so all points $[\DR(e)]$ corresponding to different choices of $s$ lie in
     the same irreducible component of $\Hlong{d}{\mathbb{A}^4}$.
    \begin{theorem}\label{ref:maintheoremexamples:thm}
        For all $e\geq 2$ the point $[\DR(e)]\in \Hlong{d}{\mathbb{A}^4}$ is a
        smooth point on an elementary component $\compo(e)$ of dimension
        $4\deg \DR(e) - (e-1)(e+5) = e^{4} + 2e^{3} - 4e + 1$.
    \end{theorem}
    The component $\compo(3)$ gives a negative answer to the following open question, stated
    at~\cite{aimpl}: \emph{Is the Gr\"obner fan a discrete invariant that distinguishes the
        components of $\Hlong{d}{\mathbb{A}^n}$?} Namely,
        we construct a curve $C \subset \Hlong{35}{\mathbb{A}^4}$ such that all subschemes corresponding to points
        of $C$ share the same Gr\"obner fan with respect to the standard torus
        action,
        and a general point of $C$ is a smooth point on $\compo(3)$, while the special point on $C$ lies in the
        intersection of $\compo(3)$ and another component, see Example~\ref{ex:groebnerfans}.

    A central idea of the proof of
    Theorem~\ref{ref:TNTrelevance:intro} is to
    consider a scheme $\HplusN$ that is a generalization of the \BBname{}
    decomposition for the scheme $\HshortN$ with the $\Gmult$-action coming
    from the standard $\Gmult$-action on $\mathbb{A}^n$. Intuitively, the
    scheme
    $\HplusN$ parameterizes families that have a limit at infinity,
    see~Section~\ref{sec:BBdecomposition} for details. In particular:
    \begin{itemize}\setlength{\itemsep}{0pt}
        \item the $\kk$-points of $\HplusN$ correspond to subschemes $\DR \subset
            \mathbb{A}^n$ supported at the origin,
        \item the tangent space to $\HplusN$ at a $\kk$-point $[\DR]$ is $\Hom(\IR,
            \OR)_{\geq 0}$,
        \item deformations of $[\DR]\in\HplusN$ have an obstruction space $T^2(\DR)_{\geq 0}$.
    \end{itemize}
    The scheme $\HplusN$
    comes with a forget-about-the-limit-and-translate map
    \begin{equation}\label{eq:thetadef}
        \theta\colon\HplusN \times \mathbb{A}^n \to \HshortN,
    \end{equation}
    A finite subscheme $\DR$ has trivial negative tangents if and only if the tangent map $d\theta$ is surjective at $([\DR],
    0)$ and in this case $\theta$ is an open immersion onto a neighbourhood of
    $\DR$; see Theorem~\ref{ref:thetaetale:thm}.
    The proof of Theorem~\ref{ref:TNTrelevance:intro}\ref{it:ElementaryImpliesTNT} crucially depends on
    representability of $\HplusN$ by a finite type scheme.

    Apart from finding new elementary components of $\HshortN$ we also investigate its
    singularities. It is widely expected that $\HshortN$ has arbitrary
    bad singularities, as defined in~\cite[p.~570]{Vakil_MurphyLaw}.
    However, absolutely no hard evidence for this exists. The
        following questions remain
    open.
    \begin{question}[\cite{fogarty, CEVV,
        aimpl}]\label{que:nonreducedness}
            Is $\HshortN$ always reduced?
    \end{question}
    In this paper, we do not answer Question~\ref{que:nonreducedness}
    directly, however we reduce it to an algorithmic tangent-space-level
    Conjecture~\ref{ref:guar:conjecture}. We recall some background results.
    Vakil~\cite{Vakil_MurphyLaw} proved the Murphy's Law for all Hilbert schemes
    other than the Hilbert schemes of points. His method depends
    on the information in the projective embedding and does not apply to
    the zero-dimensional case. Using~\cite{Vakil_MurphyLaw}, Erman~\cite{erman_Murphys_law_for_punctual_Hilb} proved that
    the Murphy's Law holds for $\HshortNGmult$. His idea is that
    homogeneous deformations of a high enough truncation of a cone correspond
    to deformations of this cone and the associated projective variety.
    Erman's result does not imply anything about $\HshortN$,
    as he explicitly states
    in~\cite[p.~1278]{erman_Murphys_law_for_punctual_Hilb}.
    The problem is
    that $\HshortNGmult$ is a closed, nowhere dense, subset of $\HshortN$.

    We extend Erman's idea to~$\HplusN$ and use the
    map~\eqref{eq:thetadef} to compare of $\HplusN\times
    \mathbb{A}^n$ and $\HshortN$. Our argument can be made on the infinitesimal
    level, however for clarity we will keep using the geometric language.
    The first step is to prove that $\HplusN$ satisfies Murphy's Law. This is
    done in Theorem~\ref{ref:MurphyForBBstrata:thm} by finding a smooth
    locus of the pass-to-the-limit retraction $\HplusN\to \HshortNGmult$.
    The second step is to compare $\HplusN \times \mathbb{A}^n$ with $\HshortN$
    using the map $\theta$ from~\eqref{eq:thetadef}. By
    Theorem~\ref{ref:TNTrelevance:intro}, if the subscheme $[\DR]$ has trivial negative tangents,
    then $\theta$ is an open embedding near $[\DR]$. It is not clear
    why a chosen point $[\DR]\in \HplusN$ with pathological deformation
    space should have trivial negative tangents.
    Therefore, we put forward a conjectural method of
    modifying any point of $\HshortN$ to a point having trivial negative
    tangents.
    \newcommand{\DT}{T}%
    \begin{conjecture}\label{ref:guar:conjecture}
        Let $\DS$ be a polynomial ring over $\kk$ and $I\subset \DS$
        be an ideal of regularity $r_0$. For all $r \geq r_0+2$, there exists an
        integer
        $t$, a polynomial ring $\DT = \DS[x_1,
        \ldots ,x_t]$, and a linear subspace $L \subset \DT_{r}$ such that the
        finite scheme $\DR$ given by the ideal $I\cdot \DT + L +
        \DT_{\geq r+1}$ has trivial negative tangents.
    \end{conjecture}
    For $I = 0$, the obtained ideals are compressed and in this case the conjecture was
    formulated already in~\cite{emsalem_iarrobino_small_tangent_space} with
    much more quantitative precision.
    Shafarevich~\cite{Shafarevich_Deformations_of_1de} proved this more
    precise version for $I = 0$ and $L$ spanned by quadrics. Not much more is
    known.
    Since $r \geq r_0+2$, the $\Gmult$-invariant deformations of $I$
    and $\IR$ are smoothly equivalent. By semicontinuity, if
    the conjecture holds for one particular $L\subset \DT_r$, then also for a
    general $L' \subset \DT_{r}$ of the same
    codimension; the particular choice of $L$ brings little information
    apart from ensuring that all deformations of $\DR$ are supported on a
    single point, by Theorem~\ref{ref:TNTrelevance:intro}\ref{it:TNTimpliesElementary}.
    Conjecture~\ref{ref:guar:conjecture} implies that $\HshortN$
    is highly singular, in particular non-reduced, answering
    Question~\ref{que:nonreducedness}, see Section~\ref{sec:singularities}.

    Let us mention another open question.
    It is known \cite{iarrobino_reducibility} that
    $\Hlong{d}{\mathbb{A}^3}$ is reducible for $d\geq 78$. However none of its
    components other than the smoothable one is known. Also the minimal $d$
    such that $\Hlong{d}{\mathbb{A}^3}$ is reducible is not known, except for
    the bound $12\leq d\leq
    78$, see~\cite{aimpl, DJNT}.
    By Theorem~\ref{ref:TNTrelevance:intro}, to prove that $d\leq d_0$ it is
    enough to find a degree $d_0$ subscheme having trivial negative tangents.
    \begin{question}\label{ref:reducibilityCodimThree}
        What is the smallest $d$ such that $\Hlong{d}{\mathbb{A}^3}$ is
        reducible?
    \end{question}
    Similarly, new examples of finite Gorenstein subschemes of $\mathbb{A}^4$
    and $\mathbb{A}^5$ would likely improve the bounds
    given~\cite[\S8.1]{bubu2010} and perhaps allow to extend the smoothability
    results on finite Gorenstein algebras presented in~\cite{cn10,
    cn11}. It would also be very interesting to know how the \BBname{}
    decompositions could be used to improve explicit computations on
    the Hilbert schemes~\cite{Bertone__double_generic,
    lella_roggero_Rational_Components}.

    Our focus on $\mathbb{A}^n$ is without loss of generality as
    Hilbert schemes of points on other smooth varieties have the same components,
    see~\cite[p.4]{artin_deform_of_sings} or~\cite{jabu_jelisiejew_smoothability}.
    The idea of proving smoothness using the \BBname{} decomposition is very
    general; we plan to investigate its applications to other moduli spaces,
    such as Quot schemes.
    We believe that the characteristic zero assumption in
    Theorem~\ref{ref:TNTrelevance:intro}\ref{it:ElementaryImpliesTNT} can be
    removed.
    Upon completion of this work we learned
    that the \BBname{} decomposition for algebraic spaces was constructed
    earlier by Drinfeld~\cite{Drinfeld}, by an entirely different method. His
    method was extended~\cite{jelisiejew_sienkiewicz}
    to actions of groups other than $\Gmult$. While this paper was in
    review, a preprint~\cite{Jelisiejew__Pathologies} appeared, which in
    particular claims to answer Question~\ref{que:nonreducedness}.

    The organization of the paper is linear.
    In Section~\ref{sec:prelims}, we discuss all algebraic
    preliminaries. In
    Section~\ref{sec:BBdecomposition},
    we construct $\HplusN$ and, in Section~\ref{sec:obstructions},
    obstruction theories. Finally, in Section~\ref{sec:singularities}, we
    discuss singularities and Conjecture~\ref{ref:guar:conjecture}. We conclude with
    Section~\ref{sec:examples}, which contains some explicit examples.
    \vspace*{-0.5\baselineskip}

    \subsection*{Acknowledgements}

        I am very grateful to Bernt Ivar Utst{\o}l N{\o}dland and Zach Teitler
        for fruitful discussions and especially to Gary Kennedy for sharing his
        observation that $[\DR_3]$ lies on a small component, which became the
        seed of this paper. I thank J.O.~Kleppe for explaining his results and
        Laudal's approach to deformation theory. I
        also thank P. Achinger, J. Buczy{\'n}ski, D. Erman, M. Grab, M.
        Huibregtse, A. Iarrobino, A. Langer,
        R. M. Mir{\`o}-Roig, \L{}. Sienkiewicz, M. Varbaro and J.
        Wi{\'s}niewski for their helpful remarks.
        I thank the referee for many suggestions which greatly improved the
        style and presentation.
        The first ideas for this article
        appeared while collaborating on~\cite{DJNT}, during Combinatorial
        Algebraic Geometry Semester at Fields Institute.
        Macaulay2~\cite{M2}, Magma~\cite{Magma}, and Singular~\cite{Singular} were used for
        experiments.

        \section{Preliminaries}\label{sec:prelims}

        Throughout the article, $\kk$ is an
        arbitrary field,
        $\DS$ is a finitely generated $\kk$-algebra, $\ambient := \Spec \DS$
        and $\HshortNAmb = \coprod_{d} \Hlong{d}{\ambient}$ is the Hilbert
        scheme of points on $\ambient$.
        For closed subschemes $\DR \subset \DM
        \subset \ambient$ by $\IM \subset \IR \subset
        \DS$ we
        denote the respective ideals and by $\OR = \DS/\IR$, $\OM = \DS/\IM$ the
        corresponding algebras. We define $\Ires = \IR/\IM$ so that there are
        short exact sequences
        \begin{align*}
            0\to \IM \to \IR &\to \Ires \to 0,\\
            0\to \Ires\to \OM &\to \OR\to 0.
        \end{align*}
        The associated long exact sequences form the following
        commutative diagram~\eqref{eq:diagram} of $\DS$-modules with exact rows and columns. Here
        and elsewhere $\Ext := \Ext_{\DS}$ and $\Hom := \Hom_{\DS}$.
        \begin{equation}\label{eq:diagram}
            \begin{tikzcd}
                & 0\arrow[d]& 0\arrow[d] & 0\arrow[d] &\\
                0\arrow[r] & \Hom(\Ires, \Ires) \arrow[r]\arrow[d] & \Hom(\Ires, \OM) \arrow[d]\arrow[r] &
                \Hom(\Ires, \OR) \arrow[d]\arrow[r] & \Ext^1(\Ires, \Ires)
                \arrow[d]\\
                0\arrow[r] & \Hom(\IR, \Ires) \arrow[r]\arrow[d] & \Hom(\IR, \OM) \arrow[d]\arrow[r] &
                \Hom(\IR, \OR) \arrow[d]\arrow[r] & \Ext^1(\IR, \Ires) \arrow[d]\\
                0\arrow[r] & \Hom(\IM, \Ires) \arrow[r]\arrow[d] & \Hom(\IM, \OM)
                \arrow[d, "\uppleftarrow"]\arrow[r, "\righttarrow"] &
                \Hom(\IM, \OR) \arrow[d, "\upparrow"]\arrow[r] & \Ext^1(\IM,
                \Ires)\arrow[d]\\
                & \Ext^1(\Ires, \Ires) \arrow[r] & \Ext^1(\Ires,
                \OM) \arrow[r] & \Ext^1(\Ires, \OR)\arrow[r]
                & \Ext^2(\Ires, \Ires)
            \end{tikzcd}
        \end{equation}

        Now we fix assumptions and conventions regarding the torus action.
        In this paper, we eventually restrict to $\DS = \kk[x_1, \ldots ,x_n]$
        graded by $\deg x_i > 0$, so it might be helpful to have this example
        in mind. For a homogeneous ideals the definitions below are
        well-known, but non-homogeneous ideals and their $\Hom$'s are used
        substantially in the proof of
        Theorem~\ref{ref:TNTrelevance:intro}\ref{it:ElementaryImpliesTNT}.

        \newcommand{\mubar}{\overline{\mu}}%
        We fix an $\mathbb{N}$-grading on $\DS$ and assume that $\DS_0 = \kk$.
        The \emph{origin} of $\DS$ is the distinguished $\kk$-point of $\Spec \DS$ given by the ideal
        $\DS_+ = \bigoplus_{i > 0} \DS_i$.
        Consider the algebraic torus $\Gmult := \Spec \kk[t^{\pm 1}]$ and the
        affine line $\Gmultbar := \Spec
        \kk[t^{-1}]$ with the natural $\Gmult$-action.
        The grading on $\DS$ induces a $\Gmult$-action on $\ambient = \Spec
        \DS$. The origin is the unique $\Gmult$-fixed point.
        For a homogeneous ideal $\IR \subset \DS$, we denote
        \[
            \Hom(\IR, \OR)_i = \{\varphi\in \Hom(\IR,\OR)\ |\
                \varphi\left((\IR)_j\right)
                    \subset (\OR)_{i+j}\mbox{ for all } j\in \mathbb{N}\}.
        \]
        Under the above convention, we have $t\cdot \varphi = t^{-i}\varphi$
        for all $t\in \Gmult(\kk)$ and all $\varphi\in \Hom(\IR, \OR)_{i}$.

        Let $\IR \subset \DS$ be an ideal, not necessarily homogeneous, supported
        at the origin. We now define $\Hom(\IR, \OR)_{\geq 0}$ and
        $\Ext^1(\IR, \OR)_{\geq 0}$. While $\Hom_{\geq 0}$ is easy,
        $\Ext^1_{\geq 0}$ requires some care and a change of perspective.
        For all $k\geq 0$, we define the $\kk$-subspaces  $\DS_{\geq k} :=
        \bigoplus_{i\geq k} \DS_i \subset \DS$, $(\IR)_{\geq k} := \IR
        \cap \DS_{\geq k} \subset \IR$ and $(\OR)_{\geq k} = (\DS_{\geq k} +
        \IR)/\IR \subset \OR$.
        We set
        \begin{equation}\label{eq:positiveHom}
            \Hom(\IR, \OR)_{\geq 0} := \left\{ \varphi\in \Hom(\IR, \OR)\ |\
                \varphi((\IR)_{\geq k}) \subset (\OR)_{\geq k}\mbox{ for all
                }k\in \mathbb{N} \right\}.
        \end{equation}
        If the ideal $\IR$ is homogeneous, then $\Hom(\IR,
        \OR)_{\geq 0} = \bigoplus_{k\geq 0} \Hom(\IR,
        \OR)_{k}$ as expected.
        We also define
        \[
            \Hom(\IR, \OR)_{<0} := \Hom(\IR, \OR)/\Hom(\IR, \OR)_{\geq 0}.
        \]
        and again for homogeneous ideals we have $\Hom(\IR, \OR)_{<0} = \bigoplus_{k< 0} \Hom(\IR,
        \OR)_{k}$ as expected.

        The group $\Ext^1_{\geq 0}$ could be defined using the
        formula \eqref{eq:positiveHom} and constructing a filtered free
        resolution~\cite{Bjork__Filtered}. However, we have not found a
        suitable reference for the independence of the resolutions, so we take
        a different path.
        Consider the ring $\DS[t^{\pm 1}]$. Ideals of $\DS$ correspond
        bijectively to homogeneous ideals of $\DS[t^{\pm 1}]$. Explicitly, for an element $i
        \in \IR$, with $i = \sum i_k$ where $i_k\in \DS_k$, we define $i^{h} :=
        \sum t^{-k}i_k$ and $\IRhomzero := (i^{h}\ |\ i\in \IR)\DS[t^{\pm
        1}]$. The ideal $\IRhomzero \subset \DS[t^{\pm 1}]$ is the homogeneous
        ideal corresponding to $\IR$. Setting
        $\ORhomzero := \DS[t^{\pm 1}]/\IRhomzero$, we have a canonical isomorphism
        \begin{equation}\label{eq:homogenization}
            \Hom(\IR, \OR) \simeq \left(\Hom_{\DS[t^{\pm 1}]}(\IRhomzero,
            \ORhomzero)\right)_0.
        \end{equation}
        Let $\IRhom := \IRhomzero\cap \DS[t^{-1}]$ and $\ORhom :=
        \DS[t^{-1}]/\IRhom \subset \ORhomzero$. On the right side of~\eqref{eq:homogenization}, the
        condition~\eqref{eq:positiveHom} translates into $\varphi(\IRhom)
        \subset \ORhom$ and we obtain canonically
        \begin{equation}\label{eq:nonnegativetangents}
            \Hom(\IR, \OR)_{\geq 0} = \left(\Hom_{\DS[t^{-1}]}(\IRhom,
            \ORhom)\right)_0.
        \end{equation}
        Since $\DR$ is supported at the origin, for large enough $k$ we have
        $\DS_{\geq k} \subset \IR$. Therefore, $\DS_{\geq k}[t^{-1}] \subset
        \IRhom$ so the algebra $\ORhom$ is finite over
        $\kk[t^{-1}]$. The algebra $\ORhom$ is flat over $\kk[t^{-1}]$;
        see~\cite[p.~346]{Eisenbud}.

        Geometrically, the family $\Spec\ORhom \subset
        \ambient \times \Gmultbar\to \Gmultbar$ corresponds to a morphism $\Gmultbar
        \to \HshortNAmb$ that is the closure of the $\Gmult$-orbit of the
        point $[\DR]$.
        The
        tangent space to $[\DR]\in
        \HshortNAmb$ is isomorphic to the space of $\Gmult$-invariant vector fields on the orbit
        $\Gmult [\DR]$.  The space $\Hom(\IR, \OR)_{\geq 0}$ is the space of
        $\Gmult$-invariant vectors fields on $\Gmultbar [\DR]$. It consists of
        the vector fields in $\Hom(\IR, \OR)$ that extend from $\Gmult[\DR]$ to
        $\Gmultbar[\DR]$.
        The advantage of~\eqref{eq:nonnegativetangents}
        over~\eqref{eq:positiveHom} is that it easily extends from $\Hom$
        to $\Ext$. We define
        \[
            \Ext^1(\IR, \OR)_{\geq 0} := \left(\Ext^1_{\DS[t^{-1}]}(\IRhom,
            \ORhom)\right)_0.
        \]
        We will also use Schlessinger's $T^2$ functor, so we
        recall its construction~\cite[\S1.3]{HarDeform}.
        For a finite subscheme $\DR \subset \ambient$ we
        fix a surjection $j\colon F\to \IR$ from a free $\DS$-module $F$. Let
        $G = \ker j$ and $K \subset F$ be the submodule generated by $j(a)b - aj(b)$ for
        all $a,b\in F$. We have $K \subset G$. We define
        $T^2(\DR)$ to be the $\DS$-module $\Hom(G/K, \OR)/\Hom(F, \OR)$.
        As $\Ext^1(\IR, \OR) = \Hom(G, \OR)/\Hom(F, \OR)$, we
        have an inclusion, which is usually strict
        \[
            T^2(\DR) \subset \Ext^1(\IR, \OR).
        \]
        We also need a homogenized version. Let $\IRhomzero \subset
        \DS[t^{\pm 1}]$ and $\IRhom \subset
        \DS[t^{-1}]$ be defined as in~\eqref{eq:homogenization}. By localization we obtain
        $T^{2}(\ORhom)\into
        T^2(\DS[t^{\pm 1}]/\IRhomzero)  \simeq T^2(\OR[t^{\pm 1}])$.
        We define $T^2(\DR)_{\geq 0}$ to be $T^2(\ORhom)_0$.
        As a result we have a canonical inclusion
        \[
             T^2(\DR)_{\geq 0} = T^2(\ORhom)_0 \into
            T^2(\OR[t^{\pm1}])_0 \simeq T^2(\DR).
        \]

        \newcommand{\Verring}[1]{\DS^{(#1)}}%
        \newcommand{\fidx}{m}%
        Following~\cite[Definition~5.15]{MR3426613}, we say that a finite subscheme $\DR \subset
        \mathbb{A}^n$ is \emph{cleavable} if it is a limit of geometrically reducible
        subschemes of $\mathbb{A}^n$.
        Here and elsewhere, \emph{geometrically} means ``after base change from
        $\kk$ to its algebraic closure
        $\kkbar$''. If $\kk$ is algebraically closed, this is vacuous, while
        for non-algebraically closed fields geometric reducibility is
        better behaved than reducibility~\cite[\S3.2.2]{Quin_Liu}.
        Put differently, the subscheme $\DR$ is
        \emph{cleavable} if the point $[\DR]$ lies on a non-elementary component of
        $\HshortN$.

        Finally, we recall the notion of a $\Gmult$-limit.
        Let $\infty := \Gmultbar(\kk) \setminus
        \Gmult(\kk)$.
        For a separated scheme $X$ with a $\Gmult$-action and a $\kk$-point
        $x\in X$, we say that the orbit of $x$ has a \emph{limit at infinity}, if the
        orbit map $\mu\colon\Gmult \ni t\to t\cdot x \in X$ extends to
        $\mubar\colon\Gmultbar \to X$. This extension is unique and we denote
        the point
        $\mubar(\infty)\in X$ by $\lim_{t\to \infty} t\cdot x$. When $X$
        is proper, the limit always exists.

        \section{The \BBname{} decomposition}\label{sec:BBdecomposition}
    \newcommand{\HSturm}{\mathcal{H}S}%
    \newcommand{\HSturmFlag}{\HSturm{}Flag}%
    \newcommand{\familY}{\mathcal{Y}}%
    \newcommand{\ambientext}{\ambient'}%

    The \BBname{} decomposition in its classical
    version~\cite[Theorem~4.3]{BialynickiBirula__decomposition} applies to a smooth
    and proper
    variety $X$ with a $\Gmult$-action. In this setup, the locus
    $X^{\Gmult}$ is also smooth and, for each its component $Y_i \subset
    X^{\Gmult}$, the
    \emph{decomposition} associated to $Y_i$ is a smooth locally closed
    subscheme of $X$
    defined by
    \[
        Y^+_i := \{x\in X \ |\ \lim_{t\to \infty} t\cdot x \mbox{ exists and
        lies in } Y_i\}.
    \]
    The \emph{\BBname{} decomposition} is $Y^+ = \coprod_i Y_i^+$.
    In this section, we generalize the decomposition to
    the case of Hilbert scheme of points
    $\HshortNAmb$, which is singular and non-proper. In contrast
    with~\cite{BialynickiBirula__decomposition}, we are interested also in
    the local theory.
    We define a functor $\HplusNAmb\colon\SchOp\to
    \Set$ by
    \begin{equation}\label{eq:Hilbplus}
        \HplusNAmb(B) := \left\{ \varphi\colon\Gmultbar\times B  \to \HshortNAmb\ |\ \varphi
            \mbox{ is } \Gmult\mbox{-equivariant}\right\}.
    \end{equation}
    Consider $\ambientext := \ambient \times \Gmultbar = \Spec (\DS[t^{-1}])$
    with its induced grading and $\Gmult$-action. We stress that the variable $t^{-1}$ has
    negative degree.
    The multigraded
    Hilbert functor
    $\HSturm\colon\SchOp\to \Set$ given by
    \[
        \HSturm(B) = \{ \famil \subset \ambientext \times B \xrightarrow{\pi} B \ |\ \famil\mbox{ is }
        \Gmult\mbox{-invariant},\ \forall_i \left(\pi_*\OO_{\famil}\right)_i\mbox{ locally
        free of finite~rank}\}
    \]
        is represented by a scheme with quasi-projective connected components~\cite[Theorem~1.1]{Haiman_Sturmfels__multigraded}. We
        denote this scheme by
        $\HSturm$. We have a natural transformation $\iota\colon \HplusNAmb \to \HSturm$. Indeed, by definition
        \[
            \HplusNAmb(B)  \simeq  \left\{ \famil \subset \ambientext \times B\ |\
                \famil \mbox{ is } \Gmult\mbox{-invariant},
                \famil\to \Gmultbar \times B\mbox{ flat, finite,
                }\Gmult\mbox{-equivariant} \right\}.
            \]
        The transformation $\iota$ assigns to an embedded family
        $\famil\to \Gmultbar \times B$ the family $\famil\to B$. Since $\famil \to \Gmultbar \times B$ is finite, flat and
        $\Gmult$-equivariant, the pushforward of $\OO_{\famil}$ is a locally
        free \mbox{$\OB$-module} with finite rank graded pieces.
    \begin{proposition}\label{ref:openembedding:prop}
        The functor $\HplusNAmb$ is represented by an open subscheme of $\HSturm$ and
        $\iota$ is the corresponding open immersion.
        Abusing notation, we use $\HplusNAmb$ to denote both the functor and the
        scheme representing it.
    \end{proposition}
    \begin{proof}
        \def\FreeMod{F}%
        \def\idxset{\Delta}%
        Choose a connected scheme $B$ and an element of $\HSturm(B)$
        corresponding to a $\Gmult$-invariant family $\famil \subset \ambientext \times B$.
        For each point $b\in B$ the family $\famil|_{b}\to b$ corresponds to an
        element of $\HSturm(b)$. Suppose a point $b\in B$ is such that the
        corresponding  element lies in
        $\iota(\HplusNAmb)$. This means that $\famil|_{b} \subset \ambientext
        \times b$ is a finite $\Gmult$-equivariant flat family over $\Gmultbar \times b$ so there
        exists a number $k$ such that the ideal of $\famil|_{b} \subset
        \ambientext \times b$ contains
        $\DS_{\geq k}\tensor_{\kk} \kappa(b)$.

        Denote by $p$ the natural $\Gmult$-equivariant affine map $\famil \to \Gmult \times B$.
        We prove that $p_*\OO_{\famil}$ is a finitely generated $\OO_{\Gmultbar \times
        B}$-module. The algebra $\kk[t^{-1}]$ is negatively-graded and
        $\DS_{\geq k}\tensor_{\kk} \kappa(b)$ is contained in the ideal of
        $\famil|_{b}$ so
        $((p_*\OO_{\famil})_{i})|_{b} = 0$ for all $i\geq k$. The
        $\OO_B$-modules $(p_*\OO_{\famil})_{i}$ are locally free for
        all $i$ and $B$ is connected, so
        $(p_*\OO_{\famil})_{i} = 0$ for all $i\geq k$. The
        ideal of $\famil \subset \ambientext \times B$ contains
        $\DS_{\geq k}\tensor_{\kk} \OO_{B}$, so the algebra $\OO_{\famil}$ is a quotient of
        $(\DS/\DS_{\geq k})\tensor_{\kk} \kk[t^{-1}]\tensor_{\kk} \OO_B$ which
        is a finitely generated $\OO_{B}[t^{-1}]$-module. Thus, the $\OO_{\Gmultbar \times
        B}$-module $p_*\OO_{\famil}$ is finitely generated. Fix a finite set
        of homogeneous generators of this module and let $\idxset\subset \mathbb{Z}$ be the finite
        set of their degrees.

        Pick a finite rank graded free $\OO_{\Gmultbar \times B}$-module $\FreeMod$
        and a graded homomorphism $r:\FreeMod \to p_*\OO_{\famil}$ such that
        $r|_{b}$ is an isomorphism.
        Denote by $r_{i}$ the $i$-th graded piece of $r$.
        The set $\idxset$ is finite and for all $i\in \idxset$ the map $(r_{i})|_{b}$ is an isomorphism, so by
        Nakayama's lemma there exists an open set $U \subset B$ such that
        for all $i\in \idxset$ the map $(r_{i})|_{U}$ is an isomorphism.
        Therefore, the map $r|_{U}$ is surjective. For all $i$ the map
        $(r_i)|_{U}$ is as surjection of locally free $\OO_B$-modules of the same
        finite rank, hence is an isomorphism. As a result, the $\OO_{\Gmult\times B}$-module
        $p_*\OO_{\famil}$ is locally free so the map $p|_{U}$ is finite
        flat so the family $\famil|_{U}\to U$ comes from an element of
        $\iota(\HplusNAmb)$.
    \end{proof}
    \begin{remark}
        Representability of the functor $\HplusNAmb$ follows also
        from~\cite[Proposition~5.3]{jelisiejew_sienkiewicz} as the Hilbert
        scheme has a covering by $\Gmult$-stable affine open
        subschemes~\cite[Chapter~18]{Miller_Sturmfels}. However, the
        embedding $\iota$ is crucial for constructing obstruction
        theories for $\HplusNAmb$ in Section~\ref{sec:obstructions}.
    \end{remark}

    We have a natural transformation $\thetazero\colon \HplusNAmb \to
    \HshortNAmb$,
    given by forgetting about the limit point. More precisely,
    for a $\Gmult$-equivariant family $\varphi\colon \Gmultbar \times B\to
    \HshortNAmb$ corresponding to a $B$-point of
    $\HplusNAmb$, we take $\thetazero(\varphi)\colon B \to \HshortNAmb$ to be the restriction
    of $\varphi$ to $\{1_{\Gmult}\} \times B$. The map $\thetazero$ is a
    \emph{monomorphism} because having $\thetazero(\varphi) =
    \varphi|_{1\times B}$ we uniquely
    recover $\varphi|_{\Gmult \times B}$ and then $\varphi$. In particular,
    the map $\thetazero$ is injective on $\kk$-points.
    By a slight abuse of notation, we identify the $\kk$-points of $\HplusNAmb$
    with their images in $\HshortNAmb$ and we denote by $[\DR]$ both the
    point of $\HshortNAmb$ and the corresponding point of $\HplusNAmb$ if the latter
    exists.

    In the following Proposition~\ref{ref:supportfixed:prop}, we classify the
    $\kk$-points of $\HshortNAmb$ lying in the image of $\thetazero$.
    The answer is very intuitive: since the grading on $\DS$ is non-negative, the action of $\Gmult$ on
    $\ambient$ is divergent, with all points except the origin going to
    infinity. Thus the only points $[\DR]\in \HshortNAmb$ for which the orbit
    has a limit at infinity are those corresponding to $\DR$ supported at the origin.
    \begin{proposition}\label{ref:supportfixed:prop}
        The morphism $\thetazero$ sends $\kk$-points of $\HplusNAmb$ bijectively
        to $\kk$-points of $\HshortNAmb$ corresponding to subschemes supported at
        the origin.
    \end{proposition}
    \begin{proof}
        Since $\thetazero$ is a monomorphism, it is injective on $\kk$-points.
        It remains to describe the $\kk$-points in its image.
        Consider a $\kk$-point $[\DR]$ corresponding to a finite subscheme $\DR \subset \ambient$
        of degree $d$. The Hilbert-Chow morphism  $\rho\colon
        \Hlong{d}{\ambient}\to \Sym^{d}
        (\ambient)$ is equivariant. The map $\rho$ is a base change of the
        projective morphism
        $\bar{\rho}\colon \Hlong{d}{\mathbb{P}} \to \Sym^d(\mathbb{P})$, where
        $\mathbb{P}$ is a
        compactification of $\ambient$, so the fibers of $\rho$ are proper.
        If $\DR \subset \ambient$ corresponds to a scheme supported at the
        origin, then $[\DR]$ lies in $\rho^{-1}(0)$, which is proper and
        $\Gmult$-invariant, hence the
        $\Gmult$-orbit of $[\DR]$ extends to $\Gmultbar\to \HplusNAmb$ and the
        point $[\DR]$ is in the image of $\thetazero$.

        Conversely, if $[\DR]$ lies in the image of
        $\thetazero$, then for every point $v\in \Supp \DR$, the map
        $\varphi$ induces an $\Gmult$-equivariant map $\varphi_v\colon
        \Gmultbar\to \ambient$ with $\varphi_v(1) = v$.
        Choose an equivariant
        embedding $\ambient \subset \mathbb{A}^N = \kk[x_1, \ldots ,x_N]$,
        where $\deg x_i > 0$ for all $i$.
        For $v = (v_1, \ldots ,v_N)\in \mathbb{A}^N$ we have
        \[
            t\cdot v = (t^{\deg x_1} v_1, \ldots , t^{\deg x_N}v_N),
        \]
        so that $\varphi_v$ exists only if $v = (0, \ldots , 0)$.
    \end{proof}
    \newcommand{\maxR}{\mathfrak{m}_{\DR}}%
    \newcommand{\grR}{\gr(\OR)}%
    \newcommand{\ORcal}{\OO_{\mathcal{\DR}}}%
    \newcommand{\IRcal}{\II_{\mathcal{\DR}}}%
    \newcommand{\maxRcal}{\mathfrak{m}_{A}}%
    \newcommand{\grRcal}{\gr(\ORcal)}%
    Our strategy is to gain knowledge about $\HshortNAmb$ by an analysis of
    $\HplusNAmb$
    and the map $\thetazero\colon \HplusNAmb \to \HshortNAmb$. This is done using obstruction theories in the next section.

    We conclude this section by providing a purely algebraic description of
    $\HplusNAmb$.
    For a $\kk$-algebra $A$ and a quotient $\ORcal = \DS \tensor A/\IRcal$, we
    have a natural filtration on $\ORcal$, where $(\ORcal)_{\geq i}$ is the
    image of $\DS_{\geq i} \tensor A$. Define $\grRcal = \bigoplus
    (\ORcal)_{\geq i}/(\ORcal)_{\geq i+1}$. The scheme $\HplusNAmb$ represents
    the functor
    \[
        \left(\HplusNAmb\right)(\Spec(A)) = \left\{ \ORcal = \DS\tensor A/\IRcal\ |\
            \ORcal\mbox{ and } \grRcal \mbox{ are finite, $A$-flat, and of
            degree } d\right\}.
    \]
    Flatness of $\grRcal$ is equivalent to preservation of local Hilbert
    function. The infinitesimal version of this construction appeared
    in~\cite[p.~614]{kleppe__smoothness}. Since we will not use this
    description further, we leave the details to the reader.

    \section{Obstruction theories}\label{sec:obstructions}

        \newcommand{\ArtCat}{\mathcal{A}rt\mathcal{L}oc_{\kk}}%
        \newcommand{\Tspace}{T}%
        \newcommand{\Obspace}{Ob}%
        \newcommand{\Artk}{\operatorname{Art}_{\kk}}%
        In this section we construct obstruction theories for Hilbert schemes
        of points
        and their \BBname{} decompositions. We choose a very explicit
        approach and follow the notation of~\cite[Chapter~6]{fantechi}
        and, when speaking about Schlessinger's $T^i$ functors, of~\cite{HarDeform}.

        The slogan is that the obstruction space for the \BBname{} decomposition of
        $\HshortN$ is the non-negative part of the obstruction
        space for $\HshortN$. An important feature is that the non-negative
        part frequently vanishes; in particular this happens
        in the setting of Theorem~\ref{ref:maintheoremexamples:thm}.

        Let $\Artk$ be the category of finite local $\kk$-algebras with
        residue field $\kk$. For an algebra $A\in \Artk$ a \emph{small
        extension} is an algebra $B\in \Artk$ together with a surjective
        homomorphism of algebras $f\colon B\onto A$ such that the ideal $K =
        \ker f$ is annihilated by the maximal ideal of $B$. A small extension
        gives rise to an exact sequence of $B$-modules $0\to K\to B\to A\to
        0$.

        Let $X$ be a $\kk$-scheme and $x\in X$ be a $\kk$-point.
        The \emph{deformation functor} associated to the pair $(X, x)$ is
        given on an algebra $B\in \Artk$ by
        \[
            D_X(B) =
            \left\{ \varphi\colon\Spec B\to X\ |\ \varphi(\Spec \kk) = x \right\},
        \]
        see~\cite[Section~6.1]{fantechi}.
        An \emph{obstruction theory} for $(X, x)$ is a pair of spaces
        $(\Tspace, \Obspace)$ and a collection of maps $ob_X$ such that for every small extension $0 \to K \to
        B\to A\to 0$ we have
        \[
            \begin{tikzcd}
                \Tspace\tensor_{\kk} K \arrow[r]& D_X(B) \arrow[r]& D_X(A) \arrow[r,
                "ob_X"] & \Obspace \tensor_{\kk} K,
            \end{tikzcd}
        \]
        see \cite[Definition~6.1.21]{fantechi}. We call $\Tspace$ and
        $\Obspace$ the \emph{tangent} and \emph{obstruction} space,
        respectively.
        For a morphism $\varphi\colon (X, x) \to (Y, y)$ a \emph{map of
        obstruction theories} from $(\Tspace_X, \Obspace_X)$ to
        $(\Tspace_Y, \Obspace_Y)$ consists of maps $\Tspace_{\varphi}\colon
        \Tspace_X\to \Tspace_Y$ and $\Obspace_{\varphi}\colon \Obspace_X\to
        \Obspace_Y$ such that for every small extension we have a commutative
        diagram
        \[\begin{tikzcd}
                \Tspace_X\tensor_{\kk} K \arrow[d,
                "\Tspace_{\varphi}\tensor \id_K"]\arrow[r]\arrow[d]& D_X(B) \arrow[r]\arrow[d]& D_X(A) \arrow[r,
                "ob_X"]\arrow[d] & \Obspace_X \tensor_{\kk} K\arrow[d,
                "\Obspace_{\varphi}\tensor \id_K"]\\
                \Tspace_Y\tensor_{\kk} K \arrow[r]& D_Y(B) \arrow[r]& D_Y(A) \arrow[r,
                "ob_Y"] & \Obspace_Y \tensor_{\kk} K.
            \end{tikzcd}
        \]
        We begin with recalling a natural obstruction theory of the Hilbert
        scheme of points. It employs Schlessinger's $T^2$ functor, which was
        recalled in Section~\ref{sec:prelims}. This theory appears in
        \cite[Theorem~6.4.5]{fantechi} but with a larger obstruction space.
        \begin{proposition}\label{ref:obstructionstandard:prop}
            The scheme $(\HshortN, [\DR])$ has an obstruction theory
            $(\Hom(\IR, \OR),\ T^2(\DR))$.
        \end{proposition}
        \begin{proof}
            For the tangent space and other details of the construction we
            refer to~\cite[Theorem~6.4.5]{fantechi}, which we follow closely.
            We abbreviate $\DS\tensor_{\kk} (-)$ to $\DS_{(-)}$.
            We begin by constructing the obstructions.
            Fix a small extension $0 \to K \to
            B\to A\to 0$ and a deformation $\mathcal{\DR}$ of $\DR$ over $A$.
            The ring $\OO_{\mathcal{R}}$ is a quotient of $\DS_{A}$ by an
            ideal $\II_{\mathcal{\DR}}$. We form a
            commutative Diagram~\ref{fig:obstructionSingle} with exact rows
            and columns. Its top row comes from applying
            $(-)\tensor_{\kk} K$ to $0\to
            \IR \to \DS \to \OR\to 0$, its bottom row comes from
            the deformation $\mathcal{R}$, and its column comes from applying
            $\DS_{(-)}$ to the small extension above.
            \begin{figure}[h]
                \begin{center}
                        \begin{tikzcd}
                            0 \arrow[r] & \IR \tensor_{\kk} K \arrow[r]\arrow[rd,
                            "\alpha"] & \DS_{K} \arrow[r]\arrow[d]
                            & \OR\tensor_{\kk} K \arrow[r] & 0\\
                            & & \DS_{B} \arrow[d]\arrow[rd, "\beta"] &&\\
                            0 \arrow[r] & \II_{\DR} \arrow[r] & \DS_{A}
                            \arrow[r] & \OO_{\mathcal{\DR}} \arrow[r] & 0
                        \end{tikzcd}
                \end{center}
                \caption{Constructing obstruction}
                \label{fig:obstructionSingle}
            \end{figure}
            The subquotient $\ker \beta/\im \alpha$ is an $\DS_{A}$-module. The
            obstruction class $ob\in \Ext^1_{\DS}(\IR, \OR)$
            is defined as the extension
            \begin{equation}\label{eq:obstructionSingle}
                0 \to \OO_{\DR} \tensor_{\kk} K \to \frac{\ker \beta}{\im
                \alpha} \tensor_{A} \kk \to \II_{\mathcal{R}}
                \tensor_{A} \kk \to 0.
            \end{equation}
            We show that this element lies in $T^2(\DR)$.
            Let $i_1, \ldots , i_r$ be the generators of $\IR$. Fix a rank
            $r$ free $\DS$-module $F$ with basis $e_1, \ldots ,e_r$ and a
            surjection $j\colon F\to I$ given by $j(e_a) = i_a$. Let
            $G= \ker j$. After a choice of lifting $\gamma\colon F\to \frac{\ker \beta}{\im
                \alpha} \tensor_{A} \kk$ we obtain the commutative diagram
                with exact rows
            \[
                \begin{tikzcd}
                    0\arrow[r] & G \arrow[r]\arrow[d] & F \arrow[r, "j"]\arrow[d,
                    "\gamma"] & \IR
                    \arrow[r]\arrow[d, "="] & 0\\
                    0 \arrow[r] & \OO_{\DR} \tensor_{\kk} K \arrow[r] & \frac{\ker \beta}{\im
                    \alpha} \tensor_{A} \kk \arrow[r] & \II_{\mathcal{R}}
                    \tensor_{A} \kk \arrow[r] & 0
                \end{tikzcd}
            \]
            The obstruction $ob$ lies in $T^2(\DR)$ if and only if $\gamma(i_a e_b - i_b e_a)
            =0$ for all $1\leq a\leq r$ and all $1\leq b \leq r$. Consider the
            following commutative diagram of surjections of
            $\DS_{B}$-modules
            \[
                \begin{tikzcd}
                    \ker \beta \arrow[d,
                    twoheadrightarrow, "\pi"]\arrow[r, twoheadrightarrow] &
                    \II_{\mathcal{R}}\arrow[d, twoheadrightarrow, "(-)\tensor_{A} \kk"] \\
                    \frac{\ker \beta}{\im \alpha}\tensor_{A} \kk \arrow[r,
                    twoheadrightarrow] & \IR
                \end{tikzcd}
            \]
            Choose lifts $j_1, \ldots , j_r\in \ker \beta$ of
            $i_1, \ldots , i_r\in \IR$ respectively. For every element $f$ of
            $(\ker \beta/\im \alpha)\tensor_{A} \kk$ we have $\pi(j_a)f =
            i_af$. Choose $\gamma$ so
            that $\gamma(e_a) = \pi(j_a)$ for each $a=1, \ldots ,r$.
            It follows that $\gamma(i_a e_b - i_b e_a) = i_a\gamma(e_b) - i_b\gamma(e_a)
            = \pi(j_a)\pi(j_b) - \pi(j_b)\pi(j_a) = 0$. Hence, we see that $\gamma(K) = 0$,
            so $ob\in T^2(\DR)$, which concludes the proof.
        \end{proof}

        Consider the finite subscheme $\DR \subset \ambient$ supported at the origin. We show that the scheme
        $(\HplusNAmb, [\DR])$ has an obstruction theory with  obstruction
        space $T^2(\DR)_{\geq 0}$. As explained in the introduction, the
        restriction from $T^2(\DR)$ to $T^2(\DR)_{\geq 0}$ is crucial for
        proving smoothness.
        \begin{theorem}\label{ref:positivetangentobstruction:thm}
            The scheme $(\HplusNAmb, [\DR])$ has  an obstruction
            theory with tangent space $\Hom(\IR, \OR)_{\geq 0}$ and
            obstruction space $T^2(\DR)_{\geq 0} \subset \Ext^1(\IR,
            \OR)_{\geq 0}$.
        \end{theorem}
        \begin{proof}
            The open embedding $\iota$ from
            Proposition~\ref{ref:openembedding:prop} sends $[\DR]$ to a
            family $\famil \subset \ambient \times \Gmultbar$ given by the
            ideal $\IRhom$, see Section~\ref{sec:prelims}. Applying the
            argument of Proposition~\ref{ref:obstructionstandard:prop} to
            $(\HSturm, [\famil])$ and taking into account $\Gmult$-invariance,
            we obtain an obstruction theory
            \[
                \left(\Hom_{\DS[t^{-1}]}(\IRhom, \ORhom)\right)_0,\quad
                (T^2_{\famil})_0.
            \]
            These spaces are isomorphic to $\Hom_{\DS}(\IR, \OR)_{\geq 0}$ and
            $T^2(\DR)_{\geq 0}$ respectively, see Section~\ref{sec:prelims}.
        \end{proof}

        \newcommand{\partialbar}{\overline{\partial}}%
        Below we consider $\ambient$ equal to $\mathbb{A}^n$ with positive
        grading. In algebraic terms, we consider $\ambient = \Spec \DS$ for a
        graded polynomial ring $\DS = \kk[x_1, \ldots
        ,x_n]$ with $\deg(x_i) > 0$ for all $i$.
        Recall from Section~\ref{sec:BBdecomposition} the natural
        forget-about-the-limit map $\thetazero\colon \HplusN\to
        \HshortN$. We form the map
        \[
            \theta\colon\HplusN \times \mathbb{A}^n \to \HshortN,
        \]
        that sends $([\DR], v)$ to the subscheme $[\DR]$ translated by
        the vector $v$.
        The map $\theta$ is the forget-about-the-limit-and-translate map
        defined in the introduction.

        \begin{lemma}\label{ref:universallyinjective:lem}
            The morphism $\theta$ is injective on $K$ points for all fields
            $K \supset \kk$.
        \end{lemma}
        \begin{proof}
            After base change to $K$, we may assume $K = \kk$.
            Let $([\DR], v)$ be a $\kk$-point of $\HplusN \times
            \mathbb{A}^n$. Then $v$ is the support of $\theta([\DR], v)$.
            From $\theta([\DR], v)$ we recover $\thetazero([\DR])$ as the
            scheme $\theta([\DR], v)$ translated by $-v$. But $\thetazero$ is
            a monomorphism, hence is injective on $\kk$-points. This concludes
            the proof.
        \end{proof}
        \begin{remark}
            The map $\theta$ is a monomorphism of schemes when $\kk$ has
            characteristic zero. It is not a monomorphism for $\charr\kk = p >
            0$, because the action of $\mathbb{A}^n$ on $\HshortN$ is not free;
            for example the stabilizer of $(x_1^p, x_2, \ldots ,x_n)$ is not
            reduced.
        \end{remark}
        The tangent bundle to $\mathbb{A}^n$
        is trivial, spanned by global sections corresponding to partial
        derivatives $\partial_i$. The tangent map $d\theta_{[\DR]}$
        sends $\partial_i$ to a homomorphism $\partialbar_i\colon \IR \to \OR$
        defined
        by the formula $(\partialbar_i)(s) := \partial_i(s) + \IR$.
        We obtain the linear subspace
        \[
            \langle \partialbar_1, \partialbar_2,  \ldots, \partialbar_n \rangle \subset \Hom(\IR, \OR).
        \]
        The image of $d\theta_{[\DR]}$ in $\Hom(\IR, \OR)$ is
        equal to
        $\Hom(\IR, \OR)_{\geq 0} + \langle \partialbar_i\ | 1\leq i\leq n\rangle$.

        Now we prove the key theorems comparing $\HplusN$ and $\HshortN$.
        Following~\cite[Definition~1.15]{MR3426613}, we say that a finite subscheme $\DR
        \subset \mathbb{A}^n$ is \emph{cleavable}, if it is a limit of
        geometrically reducible subschemes. Put differently, the subscheme $\DR$ is
        cleavable if and only if the point $[\DR]$ lies on a non-elementary
        component of $\HshortN$.

        \begin{theorem}\label{ref:thetaetale:thm}
            Let $\DR \subset \mathbb{A}^n$ be a subscheme supported at the
            origin and
            with trivial negative tangents. Then the
            map
            \[
                \theta\colon\HplusN \times \mathbb{A}^n \to \HshortN
            \]
            is an open embedding of a
            local neighbourhood of $([\DR], 0)$ into $\HshortN$. Hence, every
            component of $\HshortN$
            containing $[\DR]$ is elementary. In particular, if $\DR\neq \Spec
            \kk$, then the scheme $\DR$
            is not smoothable or cleavable.
        \end{theorem}
        \goodbreak%
        \begin{proof}
            \def\Hshort{\mathcal{H}}%
            \def\Hplus{\Hshort^{+}}%
            \def\Hplusextended{\Hplus \times \mathbb{A}^n}%
            For brevity, we denote $\Hshort := \HshortN$ and $\Hplus := \HplusN$.
            Since $\DR$ has trivial negative tangents, the tangent map
            $d\theta_{[\DR]}$ is surjective.
            Fix a small extension $0\to K\to B\to A\to 0$.
            By Theorem~\ref{ref:positivetangentobstruction:thm}, the map
            $\theta$ induces a map of obstruction theories
            \[\begin{tikzcd}
                    \Tspace_{\Hplusextended, [\DR]} \tensor_{\kk} K \arrow[d,
                        "d\theta_{[\DR]}", two heads]\arrow[r]\arrow[d]& (\Hplusextended)(B)
                \arrow[r]\arrow[d]& (\Hplusextended)(A) \arrow[r]\arrow[d] &
                \Ext(\IR, \OR)_{\geq 0} \tensor_{\kk} K\arrow[d, hook,
                "\Obspace_{\theta}"]\\
                    \Tspace_{\Hshort}\tensor_{\kk} K \arrow[r]& \Hshort(B)
                    \arrow[r]& \Hshort(A) \arrow[r] & \Ext(\IR, \OR) \tensor_{\kk} K.
                \end{tikzcd}
            \]
            By a diagram chase, the map $\theta^{\#}\colon \hat{\OO}_{\Hshort, [\DR]} \to
            \hat{\OO}_{\Hplus, [\DR]}$ satisfies infinitesimal lifting.
            Consequently, the map $\theta$ is smooth at $[\DR]$. By
            Lemma~\ref{ref:universallyinjective:lem}, this map is universally
            injective~\cite[Tag~01S2]{stacks_project}, hence it is \'etale at $[\DR]$ so it is
            \'etale at some neighbourhood $U$ of $[\DR]$. But then
            $\theta|_{U}\colon U\to \HshortN$ is an open
            immersion~\cite[Tag~02LC]{stacks_project} and
            so by
            Proposition~\ref{ref:supportfixed:prop} every component of
            $\HshortN$ containing $[\DR]$ is elementary.
        \end{proof}
        \begin{corollary}\label{ref:elementary:cor}
            Let $\DR \subset \mathbb{A}^n$ be a finite subscheme
            supported at the origin and with trivial negative tangents.
            Suppose further that $T^2(\DR)_{\geq 0} = 0$ or $\Ext^{1}(\IR, \OR)_{\geq 0} =
            0$. Then $[\DR]$ is a smooth point of $\HshortN$ lying on a
            unique elementary component of dimension $n + \dim_{[\DR]}
            \HplusN =
            \dim_{\kk} \Hom(\IR, \OR)$.
        \end{corollary}
        \begin{proof}
            By Theorem~\ref{ref:positivetangentobstruction:thm}, the point
            $[\DR]\in \HplusN$ is smooth. By Theorem~\ref{ref:thetaetale:thm} the
            map $\theta$ is an open immersion near $[\DR]$ so $[\DR]\in \HshortN$ is smooth as well.
            Other assertions follow immediately.
        \end{proof}
        In the following corollary we show that having trivial
        negative tangents can be, in most cases, deduced from the Hilbert function
        of the tangent space.
        \begin{corollary}\label{ref:independence:cor}
            Let $\DS$ be standard graded and let $\DR \subset \mathbb{A}^n$ be a finite subscheme supported at
            the origin. Suppose that the characteristic of $\kk$ is zero or the ideal $\IR$ is
            homogeneous. Then the subscheme $\DR$ has trivial negative tangents if and
            only if $\dim_{\kk} \Hom(\IR, \OR)_{<0} = n$.
        \end{corollary}
        \begin{proof}
            \newcommand{\hasse}[2]{\partial_{#1}^{[#2]}}%
            If $\DR$ has trivial negative tangents, then by
            Theorem~\ref{ref:thetaetale:thm} the map
            $d\theta_{[\DR]}$ is bijective, so $\dim_{\kk} \Hom(\IR,
            \OR)_{<0} =
            \dim_{\kk} \spann{\partial_1, \ldots , \partial_n} = n$.
            Suppose $\dim_{\kk} \Hom(\IR, \OR)_{<0} = n$.
            We are to prove that $\partialbar_1, \partialbar_2, \ldots,
            \partialbar_n$ are linearly independent elements of $\Hom(\IR, \OR)/\Hom(\IR, \OR)_{\geq 0}$. Suppose it is not so. After a coordinate change we see that
            $\partialbar_1\in \Hom(\IR, \OR)_{\geq 0}$. By definition of
            $\Hom_{\geq 0}$, the operator $\partial_1$ satisfies for all $k$
            the containment
            \begin{equation}\label{eq:positivederivation}
                \partial_1((\IR)_{\geq k}) \subset \IR + \DS_{\geq k}.
            \end{equation}

            Suppose that the characteristic is zero and take a smallest $k$
            such that $(\IR)_{\geq k} = \DS_{\geq k}$. Then
            $\partial_1((\IR)_{\geq k}) = \DS_{\geq k-1}$, so $\DS_{\geq k-1}
            \subset \IR$, a contradiction with the choice of $k$.

            Suppose that $\IR$ is homogeneous. In this
            case Equation~\eqref{eq:positivederivation} implies that $\partial_1(\IR) \subset
            \IR$ so that $\charr \kk = p$ is positive and $\IR$ is generated by
            elements of $\kk[x_1^p, x_2, \ldots ,x_n]\cap\IR$.
            Recall that for all $1\leq i\leq n$ the derivation $\partial_i$ is a first element of a
            sequence of differential operators $\hasse{i}{j}$ on $\DS$, called
            Hasse derivatives. They are defined first as operators on
            $\mathbb{Z}[x_1, \ldots ,x_n]$ by
            \[
                \hasse{i}{j} := \frac{1}{j!}\frac{\partial^j}{\partial
                x_i^j}
            \]
            and then descended to $\DS$ by $\kk$-linearity.  For all $s\in
            \mathbb{N}$ and $f, g\in \DS$ they satisfy the
            identity $\hasse{i}{s}(fg) = \sum_{j=0}^s
            \hasse{i}{j}(f)\hasse{i}{s-j}(g)$.

            Let $Q = (q_1, \ldots , q_n)$ be the lex-largest sequence of powers of
            $p$ such that after some coordinate change the ideal $\IR$ is generated by $\IR\cap \kk[x_1^{q_1}, \ldots
            , x_n^{q_n}]$. Let $\DS' = \kk[x_1^{q_1}, \ldots , x_n^{q_n}]$.
            For all $i$ and $j < q_{i}$ we have $\hasse{i}{j}(\IR) \subset
            \IR$ so by the identity above the operator $\hasse{i}{q_i}$ descents to a
            $\DS$-module homomorphism $\overline{\hasse{i}{q_i}}\colon \IR\to \OR$ of
            degree $-q_i$.
            The operator $\hasse{i}{q_i}$ acts on $\DS'$ as a partial
            derivative with respect to $i$-th coordinate.
            We claim that the homomorphisms $\{\hasse{i}{q_i}\ |\ i=1, \ldots ,
            n\}$ are linearly independent. Suppose not. By homogeneity,
            the dependence occurs
            between operators $\{\overline{\hasse{i}{q_i}}\ |\ i=1, \ldots
            ,n,\ q_i = q\}$ for some fixed $q$. After a coordinate change we have
            $\overline{\hasse{i}{q_i}} = 0$.  If $\hasse{i}{q_i}(\IR) \subset
            \IR$, then the ideal $\IR\cap \DS'$ is preserved by the
            \mbox{$i$-th} partial derivative so the sequence $Q$ is not
            maximal, a contradiction. We conclude that all homomorphisms
            $\overline{\hasse{i}{q_i}}$ are linearly independent.
            We claim that also the elements
            \begin{equation}\label{eq:minus}
                \left\{x_i^{j}\overline{\hasse{i}{q_i}}\ |\ i=1, 2, \ldots ,n,\ 
            j=0, 1, \ldots , q_i-1 \right\}
            \end{equation}
            are linearly independent in
            $\Hom(\IR, \DS/\IR)_{<0}$. Suppose it is not so.
            Since $\IR$ is generated by $\IR\cap \DS'$, the quotient $\DS/\IR$
            is naturally graded by $G = \mathbb{Z}/q_1\times \mathbb{Z}/q_2 \times  \ldots \times
            \mathbb{Z}/q_n$. The homomorphism $x_i^j \overline{\hasse{i}{q_i}}$ is non-zero
            and has $G$-degree $(0, \ldots , 0, j, 0, \ldots , 0)$, so the linear dependence may occur only between
            elements of $\{\overline{\hasse{i}{q_i}}\ |\ i=1, \ldots ,n\}$.
            But that was excluded above.
            Therefore, there are no linear dependencies and the set~\eqref{eq:minus} is a basis of a $(q_1 +
            \ldots  + q_n)$-dimensional subspace of $\Hom(\IR, \OR)_{<0}$. But
            $q_1 +  \ldots  + q_n \geq p + (n-1) > n$, a contradiction.
        \end{proof}
        \begin{remark}
            We do not know whether the equivalence of Lemma~\ref{ref:independence:cor} holds
            also for non-homogeneous ideals in positive
            characteristic, though we would guess so.
        \end{remark}
        \begin{theorem}\label{ref:generictrivialtangents:thm}
            Let $\compo \subset \HshortN$ be a irreducible component. Suppose that
            $\compo$ is generically reduced and that $\kk$ has characteristic
            zero. Then $\compo$ is
            elementary if and only if a general point of $\compo$ has trivial
            negative tangents.
        \end{theorem}
        \begin{proof}
            \def\compoplus{\mathcal{W}}%
            If any point of $\compo$ has trivial negative
            tangents, then $\compo$ is elementary by
            Theorem~\ref{ref:thetaetale:thm}. Conversely, suppose that
            $\compo$ is elementary, so it lies in the image of $\theta$. Take
            an irreducible component $\compoplus$ of $\HplusN \times \mathbb{A}^n$
            which dominates $\compo$.
            By assumptions, a general point of $\compo$ is smooth. Choose a
            point $[\DR]\in \theta(\compoplus)$ which is a
            smooth point of $\HshortN$.
            The characteristic is zero, so the
            tangent map
            \[
                d\theta_{[\DR]}\colon T_{\HplusN \times \mathbb{A}^n,
            [\DR]}\to T_{\HshortN, [\DR]}
            \]
            is injective at $[\DR]$, see proof of
            Corollary~\ref{ref:independence:cor}.
            Since $[\DR]$ is smooth, we have $\dim T_{\HshortN, [\DR]} =
            \dim_{[\DR]} \HshortN = \dim
            \compo$.
            Since $\theta$ is dominating near $[\DR]$, we have $\dim T_{\HplusN \times \mathbb{A}^n,
            [\DR]} \geq \dim_{[\DR]}
            (\HplusN\times \mathbb{A}^n) \geq \dim\compo$.
            The source of the injective map $d\theta_{[\DR]}$ has dimension at least equal to the
            dimension of the target, so this map is surjective as well.
        \end{proof}
        \begin{example}
            An earlier version of this paper incorrectly stated that in the proof
            of Theorem~\ref{ref:generictrivialtangents:thm} we may take
            $[\DR]\in \compo$ to be any smooth point of $\HshortN$.
            This is false. For example take $I = (x_1, x_2)^2 + (x_3, x_4)^2 +
            (x_1x_3 + x_2x_4)\subset \DS = \mathbb{Q}[x_1, x_2, x_3, x_4]$, which is
            a smooth point on an elementary component by~\cite{CEVV} or
            Theorem~\ref{ref:maintheoremexamples:thm}. Introduce a grading on
            $\DS$ by $\deg(x_1) = \deg(x_3) = 3$
            and $\deg(x_2) = \deg(x_4) = 1$. Then a direct computation shows
            that $\Spec(\DS/I)$ does not have trivial negative tangents.
            In this example, the grading on $\DS$ is non-standard. We do not know whether a
            similar example exists for the standard grading.
        \end{example}

        \newcommand{\iotaflag}{\iota_{\mathrm{flag}}}%
        \newcommand{\HilbFlag}{\operatorname{HilbFlag}}%
        \newcommand{\HilbFlagplus}{\operatorname{HilbFlag}^+}%
        \newcommand{\Hilbm}{\OHilb_{\DM}}%
        \newcommand{\Hilbd}{\OHilb_{\DR}}%
        \newcommand{\Hilbmplus}{\Hilbm^{+}}%
        \newcommand{\Hilbdplus}{\Hilbd^{+}}%

        Let us pass to the flag case. We return to the slightly more general
        setup: we consider subschemes of $\ambient$ instead of just
        $\mathbb{A}^n$.
        Consider the flag Hilbert scheme $\HilbFlag$, parametrizing pairs $\DR
        \subset \DM$ of subschemes of
        $\ambient$, see~\cite[Section~4.5]{Sernesi__Deformations}. As in Proposition~\ref{ref:openembedding:prop}, we see
        that its \BBname{} decomposition
        $\HilbFlagplus$ exists and is embedded by a map $\iotaflag$ as an open
        subset of a multigraded flag Hilbert scheme $\HSturmFlag$ defined
        functorially by
        \begin{align}\label{eq:HSturmflag:def}
            \HSturmFlag(B) = \{ \familY \subset \famil\subset \ambientext
                \times B\xrightarrow{\pi} B \ |\ &\familY, \famil\mbox{ are }
            \Gmult\mbox{-invariant},\\ &\forall_i
            \left(\pi_*\OO_{\familY}\right)_i\mbox { and
            }\left(\pi_*\OO_{\famil}\right)_i \mbox{ are locally
            free of finite rank}\}.\nonumber
        \end{align}
        The scheme $\HilbFlag$ comes with forgetful maps $\pi_{\DM}([\DR \subset \DM])
        = [\DM]$ and $\pi_{\DR}([\DR\subset\DM]) = [\DR]$.
        We obtain a diagram of schemes
        \begin{equation}
            \label{eq:comparisonmaps}
            \begin{tikzcd}
                \Hilbmplus \arrow[d] &\HilbFlagplus \arrow[d]\arrow[r,
                "\pi_{\DR}^+"]\arrow[l,
                "\pi_{\DM}^+"'] &
                \Hilbdplus\arrow[d]\\
                \Hilbm &\HilbFlag \arrow[r, "\pi_{\DR}"]\arrow[l,
                "\pi_{\DM}"'] &
                \Hilbd,
            \end{tikzcd}
        \end{equation}
        where $\Hilbd = \Hilbm = \HshortNAmb$ and the subscript
        indicates the $\kk$-point of interest.

        Fix subschemes $\DR \subset \DM \subset\ambient$
        supported at the
        origin and denote by $p=[\DR \subset \DM]$ the obtained $\kk$-point of
        $\HilbFlagplus$.
        We construct an obstruction theory
        for the scheme $(\HilbFlagplus, p)$ and maps between the obstruction
        theories of the schemes of the upper row of Diagram~\eqref{eq:comparisonmaps}.

        \newcommand{\Defpair}{D_{\HSturmFlag}}%
        \newcommand{\DefM}{D_{\Hilbmplus}}%
        \newcommand{\DefR}{D_{\Hilbdplus}}%
        \newcommand{\Tspaceflag}{\Tspace_{\mathrm{flag}}}%
        \newcommand{\Obspaceflag}{\Obspace_{\mathrm{flag}}}%

        \begin{theorem}\label{ref:obstructionmain:thm}
            Suppose that the map $\righttarrow_{\geq 0}$ in Diagram~\eqref{eq:diagram} is surjective. The
            pointed scheme $(\HilbFlagplus, [\DR \subset \DM])$ has an
            obstruction theory $(\Tspaceflag, \Obspaceflag)$ given by
            the following pullback diagrams
            \[
                \begin{tikzpicture}[commutative diagrams/every diagram]
                    \matrix[matrix of math nodes, name=m, commutative
                    diagrams/every cell]{
                        \Tspaceflag &     \Hom(\IR, \OR)_{\geq 0}\\
                        \Hom(\IM, \OM)_{\geq 0} & \Hom(\IM,
                        \OR)_{\geq 0}\\
                    };
                    \path[commutative diagrams/.cd, every arrow, every label]
                    (m-1-1) edge node {$\pi_{\DR}$} (m-1-2)
                    (m-1-1) edge [left] node {$\pi_{\DM}$} (m-2-1)
                    (m-2-1) edge (m-2-2)
                    (m-1-2) edge (m-2-2);
                    \begin{scope}[shift=($(m-1-1)!.15!(m-2-2)$)]
                        \draw[line width=0.3mm] +(-0.3, -0.3) -- + (0, -0.3) --
                        +(0,0);
                    \end{scope}
                \end{tikzpicture}
                \quad
                \begin{tikzpicture}[commutative diagrams/every diagram]
                    \matrix[matrix of math nodes, name=m, commutative
                    diagrams/every cell]{
                        \Obspaceflag &     T^2(\DR)_{\geq 0}\\
                        T^2(\DM)_{\geq 0} & \Ext^1(\IM,
                        \OR)_{\geq 0}\\
                    };
                    \path[commutative diagrams/.cd, every arrow, every label]
                    (m-1-1) edge node {$\pi_{\DR}$} (m-1-2)
                    (m-1-1) edge [left] node {$\pi_{\DM}$} (m-2-1)
                    (m-2-1) edge (m-2-2)
                    (m-1-2) edge (m-2-2);
                    \begin{scope}[shift=($(m-1-1)!.15!(m-2-2)$)]
                        \draw[line width=0.3mm] +(-0.3, -0.3) -- + (0, -0.3) --
                        +(0,0);
                    \end{scope}
                \end{tikzpicture}
            \]
            The projections $\pi_{\DM}$ and $\pi_{\DR}$ are maps of obstruction theories.
        \end{theorem}
        \begin{proof}
            \def\Sext{T}%
            \def\minId{K}%
            \def\alphaR{\alpha_R}%
            \def\alphaM{\alpha_M}%
            \def\betaR{\beta_{R}}%
            \def\betaM{\beta_{M}}%
            \def\IMdef{\II_{\mathcal{M}}}%
            \def\IRdef{\II_{\mathcal{R}}}%
            \def\OMdef{\OO_{\mathcal{M}}}%
            \def\ORdef{\OO_{\mathcal{R}}}%
            \def\IMtryw{\IM^h}%
            \def\IRtryw{\IR^h}%
            \def\OMtryw{\OM^h}%
            \def\ORtryw{\OR^h}%
            \def\eflag{e_{\mathrm{flag}}}%
            \def\IMplus{\JJ_{\mathcal{M}}}%
            \def\IRplus{\JJ_{\mathcal{R}}}%
            \def\OMplus{\mathcal{Q}_{\mathcal{M}}}%
            \def\ORplus{\mathcal{Q}_{\mathcal{R}}}%
            Using the embedding $\iotaflag$, see~\eqref{eq:HSturmflag:def}, it is enough to produce an
            obstruction theory for $\HSturmFlag$. Let $\Sext =
            \DS[t^{-1}] = H^0(\ambientext, \OO_{\ambientext})$.
            The tangent space is computed
            in~\cite[Proposition~4.5.3]{Sernesi__Deformations}.
            Our construction of the obstruction space below is a
            straightforward generalization of
            Proposition~\ref{ref:obstructionstandard:prop}.
            We abbreviate $\Sext\tensor_{\kk} (-)$ to $\Sext_{(-)}$.
            Consider a small extension $0\to \minId \to B\to A\to 0$ and
            an element $d\in \Defpair(A)$ corresponding to a
            $\Gmult$-invariant deformation
            $\mathcal{R} \subset \mathcal{M} \subset \Spec (\Sext_A)$.
            Diagrams~\ref{fig:obstructionSingle} for deformations
            $\mathcal{R}$ and $\mathcal{M}$ jointly form the commutative
            Diagram~\ref{fig:obstructions} with exact rows and columns.
            By Theorem~\ref{ref:positivetangentobstruction:thm}, the
            obstructions to deforming $\DR$ and $\DM$ are elements $e_{\DM}\in
            T^2(\DM)_{\geq 0} \subset \Ext^1(\IM, \OM)$ and $e_{\DR}\in
            T^2(\DR)_{\geq 0}\subset \Ext^1(\IR, \OR)$ respectively.
            \begin{figure}[h]
                \begin{center}
                \begin{tikzcd}[row sep=scriptsize, column sep = tiny]
                    & 0\arrow[r] & \IMtryw\tensor_{\kk} \minId\arrow[rrdd,
                    "\alphaM"', bend right=25] \arrow[ld, hook]\arrow[rr]
                    & & \Sext_{\minId}\arrow[rr]\arrow[dd]\arrow[ld, equal] &&
                    \OMtryw\tensor_{\kk}
                    \minId\arrow[r]\arrow[ld] & 0\\
                    0\arrow[r] & \IRtryw \tensor_{\kk} \minId \arrow[rr, crossing over]\arrow[rrdd,
                    "\alphaR"', bend right=25] && \Sext_{\minId} \arrow[rr, crossing
                    over]\arrow[dd,crossing over] && \ORtryw\tensor_{\kk} \minId\arrow[r] & 0\\
                    & &&& \Sext_{B}\arrow[dd]\arrow[ld, equal]\arrow[rrdd,
                    "\betaM", bend left=25] && \\
                    & && \Sext_{B}  &&& \\
                    & 0\arrow[r] & \IMdef \arrow[ld, hook]\arrow[rr] & & \Sext_{A}
                    \arrow[rr]\arrow[ld, equal] && \OMdef\arrow[r]\arrow[ld] & 0\\
                    0\arrow[r] & \IRdef \arrow[rr] && \Sext_{A}
                    \arrow[rr]\arrow[from=uu, crossing over] && \ORdef
                    \arrow[r]\arrow[from=uull, crossing over, bend
                    left=25, "\betaR"] & 0
                \end{tikzcd}
                \end{center}
                \caption{Constructing
                obstructions.}\label{fig:obstructions}
            \end{figure}
            By tracing their
            construction in~\eqref{eq:obstructionSingle} and comparing with
            Diagram~\ref{fig:obstructions}, we see that the images of both $e_{\DR}$ and
            $e_{\DM}$ in $\Ext^1(\IM, \OR)$ are canonically isomorphic to
            \[
                0 \to \ORtryw\tensor_{\kk} K \to \frac{\ker \betaM}{\im
                \alphaR}\tensor_{A} \kk \to
                \IMdef\tensor_{A} \kk
                \to 0,
            \]
            thus we obtain an obstruction class $\eflag = (e_{\DM}, e_{\DR})\in
            \Obspaceflag$.
            It remains to prove that the vanishing of $\eflag$ is necessary and
            sufficient for $d\in \Defpair(A)$ to lie in the image of $\Defpair(B)$.
            Necessity follows from
            Theorem~\ref{ref:positivetangentobstruction:thm}.
            \begin{figure}[h]
                \begin{center}
            \begin{tikzcd}[row sep=scriptsize, column sep = scriptsize]
                & \IMtryw\tensor \minId \arrow[ld, hook]\arrow[rr]\arrow[dd,
                    "i_{\DM}"
                near end]
                & & \Sext_{\minId}\arrow[rr]\arrow[dd]\arrow[ld, equal] &&
                \OMtryw\tensor
                \minId\arrow[ld]\arrow[dd]\\
                 \IRtryw \tensor \minId \arrow[dd]\arrow[rr, crossing over] &&
                 \Sext_{\minId} \arrow[rr, crossing over] && \ORtryw\tensor \minId \\
                 & \IMplus\arrow[rr, "j" near start]\arrow[dd] && \Sext_{B}\arrow[dd]\arrow[ld,
                 equal]\arrow[rr] && \OMplus\arrow[dd]\\
                 \IRplus\arrow[rr, crossing over]\arrow[dd] &&
                 \Sext_{B}\arrow[rr]\arrow[dd]\arrow[from=uu,crossing over]  &&
                 \ORplus\arrow[from=uu, crossing over] \\
                 & \IMdef \arrow[ld, hook]\arrow[rr] & & \Sext_{A}
                \arrow[rr]\arrow[ld, equal] && \OMdef\arrow[ld]\\
                \IRdef \arrow[rr] && \Sext_{A}
                \arrow[rr]\arrow[from=uu, crossing over] &&
                \ORdef\arrow[from=uu,crossing over, "s_{\DR}" near start]
                \end{tikzcd}
                \end{center}
                \caption{Obstruction equal to
                zero.}\label{fig:deformations}
            \end{figure}
            Suppose that $\eflag = 0$.
            It this case, $\pi_M(d)\in \DefM(A)$ and $\pi_{\DR}(d)\in
            \DefR(A)$ come from
            elements $d_1' \in \DefM(B)$ and $d_2' \in \DefR(B)$ respectively.
            These elements correspond to extensions $\IMplus, \IRplus
            \subset \Sext_B$, which give a commutative
            Diagram~\ref{fig:deformations} with exact rows and columns.

            To obtain an element of $\Defpair(B)$, we need to ensure that
            $\IMplus \subset \IRplus$. In other words, we need the
            induced $\Sext_B$-module
            homomorphism $f\colon\IMplus \to \ORplus$ to be zero,
            see~Diagram~\ref{fig:deformations}. Since $0 = f
            \circ i_{\DM} = s_{\DR} \circ f$, the map $f$ is induced from a $\Sext_B$-module homomorphism
            $\IMdef \to \ORtryw\tensor_{\kk} K$, which comes from a $\Sext$-module homomorphism $h\colon\IMtryw \to
            \ORtryw\tensor_{\kk} K$. By construction, $h$ lies in
            \[
                \Hom_{\Sext}(\IMtryw, \ORtryw)_{0} \tensor_{\kk} K  \simeq
                \Hom_{\DS}(\IM, \OR)_{\geq 0} \tensor_{\kk} K.
            \]
            By assumption on $\phi_{\geq 0}$, such a homomorphism
            lifts to a $\Gmult$-invariant homomorphism $\IMtryw \to \OMtryw
            \tensor_{\kk} K$ and hence
            it gives a $\Gmult$-invariant homomorphism $f\colon\IMplus \to \OMplus$.
            By \cite[Theorem~6.4.5]{fantechi} the element $d_1' - f\in
            \DefM(B)$ is another $\Gmult$-invariant extension of
            $\pi_{\DM}(d)$. We replace $\IMplus$ by the ideal
            $\IMplus'$ corresponding to $d_1' - f$. By
            a diagram chase, we check that the
            map $\IMplus'\to \ORplus$ is zero so $\IMplus'$ is contained in
            $\IRplus$ and we obtain an element of $\Defpair(B)$.
        \end{proof}

        In the remaining part of this section we concentrate on the coarse obstruction spaces
        $\Ext^1$ and not $T^2$.
        The following theorem summarizes our discussion and gives a rich
        source of smooth components of $\HshortNAmb$.
        The idea is to take a smooth point $[\DM]$, so that any obstruction
        $e = e_{\mathrm{flag}}$ from
        Theorem~\ref{ref:obstructionmain:thm} satisfies $\pi_{M}(e) = 0$ and
        so $\pi_{R}(e)$ lies in the kernel of
        $\Ext^1(\IR, \OR)_{\geq 0} \to \Ext^1(\IM, \OR)_{\geq0}$. This kernel
        vanishes in a number of cases, one of them discussed in Remark~\ref{ex:smallextensions}.

        Recall from Diagram~\eqref{eq:diagram} the homomorphisms
        $\righttarrow_{\geq 0}\colon\Hom(\IM, \OM)_{\geq 0}\to \Hom(\IM,
            \OR)_{\geq 0}$ and $\upparrow_{\geq 0}\colon \Hom(\IM,
            \OR)_{\geq 0}\to \Ext^1(\IR/\IM, \OR)_{\geq 0}$.

        \begin{theorem} \label{ref:smoothness:thm}
            Suppose that $[\DM]\in \HshortNAmb$ is a smooth $\Gmult$-invariant point and the maps
            $\righttarrow_{\geq 0}, \upparrow_{\geq 0}$
            are both surjective.
            For all $p = [\DR \subset \DM]\in \HilbFlagplus$, the
            points $[\DR]\in \HplusNAmb$ and $p\in \HilbFlagplus$
            are smooth and the map
            $\pi_{\DR}^+\colon \HilbFlagplus\to \HplusNAmb$
            is smooth at $p$. Moreover
            \begin{align}\label{eq:dimensioncomparison}
                \dim_{[\DR]} \HplusNAmb &= \dim_{[\DM]} \HplusNAmb - \dim
                \Ext^1(\Ires, \OR)_{\geq 0} \\ &+  \dim\Hom(\Ires,
                \OR)_{\geq 0}-\dim \Hom(\IM, \Ires)_{\geq 0}.\nonumber
            \end{align}
        \end{theorem}
        Before we prove Theorem~\ref{ref:smoothness:thm}, we put forward its
        main consequence.
        \begin{corollary}\label{ref:smoothnessHilb:cor}
            In the setting of Theorem~\ref{ref:smoothness:thm} assume
            additionally
            that $\ambient$ is equal to $\mathbb{A}^n$ with positive grading and that
            $\DR$ has trivial negative tangents. The point $[\DR] \in \HshortN$
            is smooth lying on an elementary component of dimension
            \begin{align}\label{eq:dimensioncomparisonsnd}\nonumber
                n + \dim_{[\DR]} \HplusN &= n + \dim\Hom(\IM,
                \OM)_{\geq 0} - \dim
                \Ext^1(\Ires, \OR)_{\geq 0} \\
                &+  \dim\Hom(\Ires,
                \OR)_{\geq 0}-\dim\Hom(\IM, \Ires)_{\geq 0}.
            \end{align}
        \end{corollary}
        \goodbreak%
        \begin{proof}[Proof of Theorem~\ref{ref:smoothness:thm}]
            \def\eflag{e_{\mathrm{flag}}}%
            We prove that every obstruction class $\eflag$ obtained in
            Theorem~\ref{ref:obstructionmain:thm} is actually a zero element
            of the obstruction group.
            We have an exact sequence
            \[
                \begin{tikzcd}
                    \Hom(\IM, \OR)_{\geq 0} \ar[r, "\upparrow_{\geq
                    0}"]&\Ext^1(\Ires, \OR)_{\geq 0}
                    \ar[r, "i"] & \Ext^1(\IR, \OR)_{\geq 0} \ar[r, "\nu"]&
                    \Ext^1(\IM, \OR)_{\geq 0}.
                \end{tikzcd}
            \]
            On the one
            hand, since $[\DM]\in \HshortNAmb$ is a smooth $\Gmult$-invariant
            $\kk$-point,
            the completed local ring $A = \hat{\OO}_{\HshortNAmb, [\DM]}$ is
            regular. Also the ring $\hat{\OO}_{\HplusNAmb, [\DM]} = A/(A_{>0})$
            is regular and so the scheme $\HplusNAmb$ is smooth
            at $[\DM]$. There are no obstructions for deforming $[\DM]\in \HplusNAmb$ so the obstruction class
            $\eflag$ lies in
            $\ker\nu$. On the other hand, by surjectivity of
            $\upparrow_{\geq 0}$, the map $i$ is zero, hence $\ker\nu = 0$, so
            $\eflag$ is zero
            and $\HilbFlagplus$ is smooth at $p$ by the infinitesimal lifting
            criterion.

            Since $\righttarrow_{\geq 0}$ is surjective, the tangent map
            $d\pi_{\DR}^+$ is surjective at $p$ so
            the morphism $\pi_{[\DR]}^+$ is smooth at $p$ and $[\DR]\in
            \Hilbdplus$
            is a smooth point, see \cite[Theorem~17.11.1d, p.~83]{ega4-4}.
            It remains to prove Equality~\eqref{eq:dimensioncomparison}.
            The claim follows from the description of tangent spaces in
            Theorem~\ref{ref:positivetangentobstruction:thm}
            and a chase on Diagram~\eqref{eq:diagram}. Indeed, since
            $\righttarrow_{\geq 0}$ and $\upparrow_{\geq 0}$ are surjective,
            we have
            \begin{align*}
                \dim \Hom(\IR, \OR)_{\geq 0} &= \dim \Hom(\Ires, \OR)_{\geq 0} + \dim \Hom(\IM,
                \OR)_{\geq 0} - \dim \Ext^1(\Ires, \OR)\\
                \dim \Hom(\IM, \OM)_{\geq 0} &= \dim \Hom(\IM, \OR)_{\geq 0} + \dim \Hom(\IM,
                \Ires)_{\geq 0}.\qedhere
            \end{align*}
        \end{proof}
        \begin{proof}[Proof of Corollary~\ref{ref:smoothnessHilb:cor}]
            Directly from Theorem~\ref{ref:smoothness:thm} and Corollary~\ref{ref:elementary:cor}.
        \end{proof}
        \begin{remark}\label{ref:kleppe_on_relative:rmk}
            Jan O.~Kleppe, during his investigation of $\Gmult$-invariant deformations of embedded
            schemes using Laudal deformation theory and flag Hilbert schemes in
            \cite[Proposition~4]{Kleppe__Max_Families}
            and~\cite[Theorem~1.3.4]{Kleppe__thesis}, constructed an obstruction space analogous to
            the one from Theorem~\ref{ref:obstructionmain:thm} and derived a
            theorem analogous to Theorem~\ref{ref:smoothness:thm}. Roughly
            speaking, he considers $(\HshortNAmbGmult,
            \righttarrow_0, \upparrow_0)$ instead of $(\HshortNAmb, \righttarrow_{\geq 0},
            \upparrow_{\geq 0})$.
        \end{remark}

        \begin{remark}\label{ex:smallextensions}
            Let us discuss the conditions of
            Corollary~\ref{ref:smoothnessHilb:cor} in a special case.
            Let $\DR \subset \DM \subset \mathbb{A}^n$ be $\Gmult$-invariant
            finite subschemes.
            Let $d$ be the maximal number such that $(\OM)_d \neq 0$ and suppose that
            $\Ires = \IR/\IM$ is concentrated in degree $d$. In this case
            \[
                \Ext^1(\Ires, \OR) \simeq\bigoplus  \Ext^1(\kk[-d], \OR) =
                \bigoplus  \Ext^1(\kk, \OR)[d]
            \]
            is concentrated in negative degrees, so the map
            $\upparrow_{\geq 0}$ is
            automatically surjective. Similarly,
            \[
                \Ext^1(\IM, \Ires)  \simeq \Ext^1(\IM,
                \kk)[-d]  \simeq \Ext^2(\OM, \kk)[-d]  \simeq \Tor_2(\OM,
                \kk)^{\vee}[-d]
            \]
            and this space is negatively graded exactly when there are no second
            syzygies of $\OM$ in degrees $\leq d$. Finally,
            $\DR$ having trivial negative tangents seems to be the most
            subtle assumption and we do not see any interesting sufficient conditions for
            it yet.
        \end{remark}

        \section{Singularities}\label{sec:singularities}

            In this section we prove that equicharacteristic Murphy's Law holds for $\HplusN$ and
            discuss Conjecture~\ref{ref:guar:conjecture}. Here,
            equicharacteristic is used to underline that we work over $\kk$,
            while Vakil~\cite{Vakil_MurphyLaw} works over
            $\mathbb{Z}$.

            Let us recall the main notions. A \emph{pointed scheme} $(X, x)$ is a
            scheme $X$ of finite type over $\kk$ together with a point
            $x\in X$. A \emph{morphism of pointed schemes} $(X, x)\to (Y, y)$ is a
            morphism of schemes $f\colon X\to Y$ such that $f(x) = y$.
            A \emph{retraction} is a pair $i\colon (Y, y)\to (X, x)$ and
            $\pi\colon (X, x)\to (Y, y)$ such that $\pi\circ i =
            \id_{X}$ and $i$ is closed immersion.

            Vakil~\cite{Vakil_MurphyLaw} defines an equivalence relation
            on pointed schemes
            by declaring $(X, x) \sim (Y, y)$ to be equivalent if
            there exists a pointed scheme $(Z, z)$ and smooth morphisms $(X,
            x)\leftarrow (Z, z) \rightarrow (Y, y)$. An equivalence
            class of $\sim$ is called an equicharacteristic \emph{singularity}.
            The equicharacteristic \emph{Murphy's Law holds for $\mathcal{M}$}
            if every equicharacteristic singularity appears on $\mathcal{M}$.

            The key to investigation of singularities of $\HplusN$ is its
            relation with $\HshortNGmult$. Namely,
            there is a functorial retraction $\HplusN \to \HshortNGmult$.
            To construct it, observe that $\HshortNGmult$ represents the functor of
            $\Gmult$-equivariant families $\varphi_0\colon B\to \HshortN$,
            where the $\Gmult$-action on $B$ is trivial. Recall
            from~\eqref{eq:Hilbplus} that $\HplusN$ represents the
            functor
            \[
                \HplusN(B) = \left\{ \varphi\colon \Gmultbar \times B \to
                    \HshortN\ |\ \varphi
                    \mbox{ is } \Gmult\mbox{-equivariant}\right\}.
            \]
            We have a functorial map $i\colon \HshortNGmult \to \HplusN$,
            which sends a family
            $\varphi_0\colon B\to \HshortN$ to $\varphi_0 \circ pr_2 \colon \Gmultbar \times
            B\to \HshortN$ and a functorial map
            $\pi\colon \HplusN \to \HshortNGmult$, which sends a family
            $\varphi\colon \Gmultbar \times B\to \HshortN$ to
            $\varphi|_{\infty \times B}\colon B\to \HshortN$. We have
            $\pi\circ i = \id$ by construction. A
            family $\varphi\colon \Gmultbar \times T \to \HshortN$ is equal to
            $i(\varphi_0)$ if and only if $\varphi|_{1 \times T} =
            \varphi|_{\infty \times T}$. Since $\HshortN$ is separated, $i$ is a closed immersion.
            For every $\kk$-point $[\DR]$ of $\HshortNGmult$, the morphism
            $\pi\colon (\HplusN, [\DR])\to (\HshortNGmult, [\DR])$ induces a
            map of obstruction theories, which is just the projection
            $T^2(R)_{\geq 0} \to T^2(R)_0$.

            \newcommand{\HshortFive}{\Hlong{\pts}{\mathbb{A}^5}}%
            \newcommand{\HplusFive}{\OHilb_{\pts}^{+}(\mathbb{A}^5)}%
            \newcommand{\HshortFiveGmult}{\OHilb_{\pts}^{\Gmult}(\mathbb{A}^5)}%
            \begin{theorem}\label{ref:MurphyForBBstrata:thm}
                The equicharacteristic Murphy's Law
                holds for $\HshortFiveGmult$ and for $\HplusFive$.
            \end{theorem}
            \begin{proof}
                \newcommand{\SingPoint}{\mathfrak{S}}%
                The proof for $\HshortFiveGmult$ is build around the ideas
                of~\cite{erman_Murphys_law_for_punctual_Hilb}, who actually
                proved that Murphy's Law holds for
                $\coprod_n \OHilb_{\pts}^{\Gmult}(\mathbb{A}^n)$. Our
                contribution in this case, if any, is the reduction to
                embedding dimension five.

                Fix an equicharacteristic singularity $\SingPoint$. First,
                by~\cite[M3]{Vakil_MurphyLaw} there is a surface $V \subset
                \mathbb{P}^4$ such that singularity class of the corresponding
                Hilbert scheme of surfaces $(\OHilb(\mathbb{P}^4), [V])$ is
                $\SingPoint$. Let $p = p(t)$ be its Hilbert polynomial and
                $\DS$ be the homogeneous coordinate ring of $\mathbb{P}^4$.
                By Gotzmann Regularity Theorem~\cite[Proposition~4.2]{Haiman_Sturmfels__multigraded}
                there exists a $d$ such that
                $\OHilb^p(\mathbb{P}^4)$ is isomorphic to the multigraded
                Hilbert scheme parameterizing deformations of pairs $(I_d,
                I_{d+1})$ such that
                $I_d \subset \DS_d$, $I_{d+1} \subset \DS_{d+1}$ and $\DS_1
                I_d \subset I_{d+1}$,
                see~\cite[Theorem~3.6 and its
                proof]{Haiman_Sturmfels__multigraded}.
                The isomorphism sends $V$ to $((I_{V})_d,
                (I_{V})_{d+1})$.
                This multigraded Hilbert scheme, in turn, is isomorphic to
                the locus in $\HshortFiveGmult$
                that parameterizes homogeneous ideals with Hilbert
                function $h$ satisfying $h(d-1) = \dim \DS_{d-1}$, $h(d) =
                p(d)$, $h(d+1) = p(d+1)$, and $h(d+2) = 0$. The isomorphism is
                given by sending $(I_d, I_{d+1})$ to the ideal
                \[
                    J = I_d \oplus I_{d+1} \oplus \DS_{\geq d+2}.
                \]
                Let $\DR = \Spec \DS/J$.  We conclude that the singularity
                type of $(\HshortFiveGmult, [\DR])$ is equal to $\SingPoint$.
                It remains to prove that $\pi\colon(\HplusFive, i([\DR]))\to
                (\HshortFiveGmult,
                [\DR])$ is smooth. This is formal. The homogeneous generators
                of $J = \IR$ all have
                degree at least $d$ and $(\OR)_{d+2} = 0$. We deduce that
                $\Ext^{1}(\IR, \OR)_{>0} = 0$. Therefore, the morphism
                $T^2(\DR)_{\geq 0}\to T^2(\DR)_0$ induced by $\pi$ is
                injective. The tangent map $d\pi\colon \Hom(\IR,
                \OR)_{\geq 0} \to \Hom(\IR, \OR)_0$ is clearly surjective.
                Consequently, the map $\pi$ is smooth at $i([\DR])$ by infinitesimal lifting
                property, see
                the proof of Theorem~\ref{ref:thetaetale:thm} or
                \cite[Section~6]{fantechi_Manetti}.
            \end{proof}
            As discussed in the introduction,
            Theorem~\ref{ref:MurphyForBBstrata:thm} does not shed light on the
            singularities of $\HshortFive$. As a caution, we present the following
            example.
            \begin{example}
                Consider $Y = \Spec\kk[x_1, x_2, x_3]$ and $\Gmult$ acting on $Y$ by
                $t\cdot (x_1, x_2, x_3) = (x_1, tx_2, t^{-1}x_3)$. The scheme $Z =
                \Spec \kk[x_1, x_2,x_3]/(x_1^2 - x_2x_3)$ is clearly reduced
                and has an action
                of $\Gmult$, but $Z^{\Gmult} = \Spec \kk[x_1]/(x_1^2)$ and
                $Z^{+} = \Spec \kk[x_1, x_3]/(x_1^2)$ are both non-reduced.
            \end{example}
            However, Theorem~\ref{ref:MurphyForBBstrata:thm} strongly
            suggests that $\HshortFive$ is non-reduced. In contrast,
            Conjecture~\ref{ref:guar:conjecture} implies that
            $\HshortN$
            is non-reduced for large enough $n$.
            \begin{proposition}\label{ref:nonreducednessOfHilbertScheme:prop}
                If Conjecture~\ref{ref:guar:conjecture} is true, then
                the Hilbert scheme of points on some $\mathbb{A}^n$ is non-reduced.
            \end{proposition}
            \begin{proof}[Sketch of proof]
                \newcommand{\SingPoint}{\mathfrak{S}}%
                Let $\SingPoint$ be a non-reduced singularity.
                Let $(I_d, I_{d+1})$ be as in the proof of
                Theorem~\ref{ref:MurphyForBBstrata:thm}. Let $I$ be the ideal
                generated by them. Let $r = d+2$. As we assume that
                Conjecture~\ref{ref:guar:conjecture}, we conclude that there
                exists a polynomial ring $\DT \supset \DS$, and a subspace $L
                \subset \DT_r$ such that $I' = I\cdot \DT + L + \DT_{\geq r+1}$
                has trivial negative tangents. One proves
                directly that the $\Gmult$-equivariant deformations of $I'$
                and $I + \DT_{\geq r+1}$ are smoothly equivalent.

                Let $\mathbb{A}^n := \Spec \DT$ and
                $\DR' = \Spec (\DT/I') \subset \mathbb{A}^n$. We scheme
                that $(\HshortNGmult, [\DR'])$ is non-reduced and $\DR'$
                has trivial negative tangents. Since $\pi\colon (\HplusN,
                [\DR'])\to (\HshortNGmult, [\DR])$ is a retraction, also
                $\HplusN$ is non-reduced at $[\DR']$. Since $\DR'$ has trivial
                tangents, the map
                $\theta\colon \HplusN \times \mathbb{A}^n\to \HshortN$ is an
                isomorphism near $[\DR']$, see
                Theorem~\ref{ref:TNTrelevance:intro}. Hence, also $[\DR'] \in \HshortN$ is
                a non-reduced point.
            \end{proof}
            The arxiv version
            (\href{https://arxiv.org/abs/1710.06124v3}{arXiv:1710.06124v3}) of
            this paper contains some observations potentially useful for the
            proof of Conjecture~\ref{ref:guar:conjecture}.

            \section{Examples}\label{sec:examples}

        In this section we describe several examples and prove
        Theorem~\ref{ref:maintheoremexamples:thm}.  This theorem follows from
        Corollary~\ref{ref:smoothnessHilb:cor} once we verify its
        assumptions in our case.
        We keep the notation from introduction: $\DS = \kk[x_1, x_2, y_1,
        y_2]$ and the subscheme $\DR(e) \subset \DM(e)$ is defined by a single form $s$
        so that $\Ires = \kk s$. The ideals $\IM$, $\IR$ are bi-graded
        with respect to
        \begin{equation*}
            \deg\left(x_1^{a_1}x_2^{a_2}y_1^{b_1}y_2^{b_2}\right) = (a_1 + a_2, b_1 +
            b_2),
        \end{equation*}
        so we will speak about forms of given bi-degree. Observe that $\OM$
        has a basis consisting of all monomials of bi-degree $(a,b)$ with
        $a,b<e$.

        \begin{proposition}\label{ref:zachshypo:prop}
            Let $R_1 \subset \Spec \kk[x_1, x_2]$ and $R_2\subset \Spec
            \kk[y_1, y_2]$ be finite schemes.
            Then the subscheme $R_1 \times R_2 \subset \ambient$ is
            smoothable and $[R_1 \times R_2]\in \Hlong{d_1d_2}{\mathbb{A}^4}$
            is a smooth point. In particular $\dim \Hom(I_{R_1 \times R_2},
            \OO_{R_1 \times R_2}) = 4\deg R_1 \deg R_2$.
        \end{proposition}
        \begin{proof}
            \def\cX{\mathcal{X}}%
            Let $d_i = \deg R_i$ for $i=1,2$.
            The Hilbert schemes $\Hlong{d_1}{\mathbb{A}^2}$ and $\Hlong{d_2}{\mathbb{A}^2}$ are irreducible~\cite{fogarty}
            so both $R_1$ and $R_2$ are smoothable. Consequently, $R_1 \times R_2$ is smoothable
            \cite[Proposition~5.12]{jabu_jelisiejew_smoothability}.
            Since $T_{\Hlong{d_1d_2}{\mathbb{A}^4}, [R_1 \times R_2]} =
            \bigoplus_i T_{\Hlong{d_i}{\mathbb{A}^2}, [R_i]}$, the point $[R_1
            \times R_2]$ is smooth.
        \end{proof}
        \begin{corollary}\label{ref:Msmoothexample:cor}
            For all $e\geq 2$, the subscheme $\DM(e)$ is smoothable and
            $[\DM(e)]\in \Hlong{\pts}{\mathbb{A}^n}$ is a smooth point.
        \end{corollary}
        \begin{proof}
            Apply Proposition~\ref{ref:zachshypo:prop} to
            $\DM(e) = \Spec \left(\kk[x_1, x_2]/(x_1, x_2)^e\right) \times \Spec
            \left(\kk[y_1, y_2]/(y_1, y_2)^e\right)$.
        \end{proof}

        We proceed to show that $\uppleftarrow$ and $\righttarrow$ from
        Diagram~\eqref{eq:diagram} are surjective, in all degrees, for $\DM = \DM(e)$ and $\DR =
        \DR(e)$. Recall that $\Ires  \simeq  \kk$, or, taking
        into account the grading, $\Ires  \simeq \kk[-(2e-2)]$.

        \begin{proposition}\label{ref:uppleftarrowexample:prop}
            In Diagram~\eqref{eq:diagram} applied to $\DR(e)
            \subset \DM(e)$, the homomorphism $\uppleftarrow$ is surjective.
        \end{proposition}
        \begin{proof}
            We abbreviate $\DM(e)$ and $\DR(e)$ to $\DM$ and $\DR$
            respectively.
            Since $\DM$ is monomial, it is straightforward to compute that
            $\Ext^1(\Ires, \OM)$ is concentrated in degree $-1$ and
            that
            \[
                \dim \Hom(\IM, \OM)_{-1} = \dim \Ext^1(\Ires, \OM)_{-1} = 2e(e+1).
            \]
            It is enough to show that $\Hom(\IR, \OM)_{-1}$ is
            zero. Let $s = \sum_{j} f_{j}y_1^{j}y_{2}^{e-1-j}$, where $f_{j} =
            \sum_i c_{ij} x_1^{i}x_2^{e-1-i}$, be the form defining $\DR(e)$
            inside $\DM(e)$, as in the introduction.
            Pick $\varphi\in \Hom(\IR, \OM)_{-1}$.
            The element $t := \varphi(s)$ of $\OM$ has degree $2e-3$, so it is
            uniquely written as $t_1 + t_2$ where $t_1, t_2\in \OM$
            have bi-degree $(e-2, e-1)$ and $(e-1, e-2)$ respectively.
            Write $t_1 = \sum_j g_j y_{1}^jy_2^{e-1-j}$, where $g_{j}\in
            \kk[x_1, x_2]_{e-1}$.
            Consider the equation
            \[
                x_1t_1 = x_1t = \varphi(x_1 s) = \sum_{j}
                \varphi(x_1f_{j})y_1^{j}y_{2}^{e-1-j}.
            \]
            The element $x_1f_{j}\in \kk[x_1, x_2]$ is a form of degree $e$, thus
            $x_1f_j\in \IM$. Analysing $\varphi|_{\IM}\in \Hom(\IM, \OM)_{-1}$ directly, we
            see that $\varphi(x_1f_{j})$ lies in the image $\kk[x_1, x_2]$ and so
            it is a form of degree $(e-1, 0)$. Therefore $x_1t_1$ and
            $\sum_{j} y_1^{j}y_{2}^{e-1-j}\varphi(x_1f_{j})$ are two forms of
            bi-degree $(e-1, e-1)$ equal modulo $\IM$.
            Comparing their coefficients next to $y_{1}^jy_2^{e-1-j}$, we see
            that $\varphi(x_1f_j) = x_1g_j\in \OM$.
            The same argument shows that $\varphi(x_2f_j) = x_2g_j$.

            Restrict $\varphi$ to a homomorphism $\varphi'\colon (x_1, x_2)^e\to
            \OM$ and extend $\varphi'$ to a degree minus one homomorphism $\varphi'\colon (x_1, x_2)^{e-1}
            \to \OM$ by imposing, for every $\lambda_{\bullet}\in \kk$, the
            condition
            \begin{equation*}
                \varphi'\left(\sum \lambda_j f_j\right) = \sum \lambda_j g_j.
            \end{equation*}
            The syzygies of $(x_1, x_2)^{e-1}$ are
            linear, so the map $\varphi'$ sends them to forms of degree $e-1$. No such form
            lies
            in $\IM$, thus $\varphi'$ lifts to an element of $\Hom((x_1,
            x_2)^{e-1}, \DS)_{-1}$. But $\Hom((x_1,
            x_2)^{e-1}, \DS)_{-1} = 0$ and so $\varphi' = 0$. Therefore, $\varphi( (x_1,
            x_2)^e) = \varphi'( (x_1, x_2)^e) = 0$. Repeating the argument
            with $y_i$ interchanged with $x_i$, we obtain $\varphi( (y_1,
            y_2)^e) = 0$, so $\varphi = 0$.
        \end{proof}

        \begin{proposition}\label{ref:righttarrowexample:prop}
            In Diagram~\eqref{eq:diagram} applied to $\DR(e)
            \subset \DM(e)$, the homomorphism $\righttarrow$ is surjective.
        \end{proposition}
        \begin{proof}
            We abbreviate $\DM(e)$ and $\DR(e)$ to $\DM$ and $\DR$
            respectively. We begin with a
            series of reductions. Let $N = (x_1, x_2)^{e} \oplus (y_1,
            y_2)^{e}$, with the surjection $N\to \IM$.
            We have $\Hom(N, \OM)  \simeq \Hom(\IM, \OM)$ and $\Hom(N, \OR)
            \simeq \Hom(\IM, \OR)$ so
            it is enough to show that $\Hom(N, \OM)\to \Hom(N, \OR)$ is
            surjective. Second, it is enough to show that for $N_0 =
            (x_1, x_2)^e$ the map
            \[
                \Phi:\Hom(N_0, \OM) \to \Hom(N_0, \OR)
            \]
            is surjective. Third, the map $\Phi$ preserves bi-degree, so we may restrict to
            homomorphisms of given bi-degree. Fourth, the generators of syzygies of $N_0$ are linear of
            bi-degree $(1, 0)$, the modules $\OM$ and $\OR$ differ only in bi-degree
            $(e-1,e-1)$ and $N_0$ is generated in bi-degree $(e, 0)$. If we consider
            homomorphisms of bi-degree $(d_1, d_2)\neq (-2, e-1)$ then the syzygies of
            $N_0$ are mapped into degree $(e, 0) + (d_1, d_2) + (1, 0) \neq (e-1,
            e-1)$ and the map $\Phi$ is an isomorphism. Hence, we restrict to homomorphisms of
            bi-degree $(-2, e-1)$. Each such homomorphism sends generators of $N_0$ to elements of
            bi-degree $(e-2, e-1)$.

            Pick a homomorphism $\varphi\in \Hom(N_0, \OR)$ of bi-degree $(-2, e-1)$.
            For each $0\leq i\leq e$ the element
            $\varphi(x_1^{e-i}x_2^{i})$ is a form of bi-degree $(e-2, e-1)$ so it can be
            uniquely lifted to a form $\varphi_i\in \DS$ of bi-degree $(e-2, e-1)$.
            Recall that $\IR = \IM + \kk s$.
            Since $\varphi$ is a homomorphism to $\OR$, the syzygies between elements
            of $N_0$ give the following relations between forms of bi-degree
            $(e-1, e-1)$:
            \begin{equation}\label{eq:syzygies}
                x_1\varphi_{i+1} - x_2 \varphi_{i} = \lambda_i s \mod \IM\quad i=0,1,
                \ldots ,e-1.
            \end{equation}
            To prove that $\varphi$ is in the image of $\Phi$ it is enough to prove that
            $\lambda_i = 0$ for all $i$.
            Since $s$ is general, for appropriate choice of a basis $f_0,
            \ldots ,f_{e-1}$ of $\kk[y_1, y_2]_{e-1}$ we have
            \[
                s = x_2^{e-1}f_0 + x_{2}^{e-2}x_1f_1 +  \ldots + x_{1}^{e-1}f_{e-1}.
            \]
            \def\ind{{k}}%
            Fix $\ind\in \{0, \ldots ,e\}$. Both sides of
            each Equation~\eqref{eq:syzygies} have the form $\sum_j r_j f_j$, where
            $r_j$ belong to the image of $\kk[x_0, x_1]_{e-1}$ in $\OM$. Extracting coefficients of
            $f_{\ind}$ from both sides, we obtain equalities in $\kk[x_1,x_2]$:
            \begin{equation}\label{eq:syzygiesrest}
                x_1\tau_{i+1} - x_2 \tau_{i} = \lambda_i x_1^{e-1-\ind}x_2^{\ind} \mod
                (x_1, x_2)^e\mbox{ for } i=0, \ldots ,e,
            \end{equation}
            where $\tau_i$ is the coefficient of $f_{\ind}$ in $\varphi_i$.
            Let $m_0 :=
            x_1^{\ind}x_{2}^{e-1-\ind}$. We have $m_0\cdot (x_1, x_2)^e
            \subset (x_1^e, x_2^e)$, hence
            Equation~\eqref{eq:syzygiesrest}
            for $i = \ind$ multiplied by $m_0$ gives
            \begin{equation}\label{eq:specialindex}
                x_1^{\ind+1}x_2^{e-1-\ind}\tau_{\ind+1} -
                x_1^{\ind}x_2^{e-\ind}\tau_{\ind} =
                \lambda_{\ind}x_1^{e-1}x_{2}^{e-1}\mod (x_1^e, x_2^e).
            \end{equation}
            For all monomials $m\in \kk[x_1, x_2]_{e-1}$ different from
            $m_0$ we have $x_1^{\ind}x_2^{e-1-\ind}\cdot m \in (x_1^e, x_2^e)$.
            Multiplying Equation~\eqref{eq:syzygiesrest} for $i=\ind-1$ by the
            monomial $x_1^{\ind-1}x_2^{e-\ind}$, we obtain
            \[
                x_1^{\ind}x_2^{e-\ind}\tau_{\ind} -
                x_1^{\ind-1}x_2^{e-\ind+1}\tau_{\ind-1} =
                x_1^{\ind-1}x_2^{e-\ind}\cdot \lambda_{k-1} x_1^{e-1-\ind}x_2^{\ind}
                = \lambda_k x_1^{e-2}x_2^e = 0\mod
                (x_1^e, x_2^e).
            \]
            Similarly, Equations~\eqref{eq:syzygiesrest} for $i=\ind-1,\ind-2,
            \ldots $ and $i=\ind+1,\ind+2, \ldots,$ give
            \begin{align}\label{eq:godownindexes}
                x_1^{\ind}x_2^{e-\ind}\tau_{\ind} &=
                x_1^{\ind-1}x_2^{e-\ind+1}\tau_{\ind-1} = \ldots
                =x_2^{e}\tau_{0} = 0\mod (x_1^e, x_2^e).\\
                \label{eq:goupindexes}
                x_1^{\ind+1}x_2^{e-1-\ind}\tau_{\ind+1} &=
                x_1^{\ind+2}x_2^{e-2-\ind}\tau_{\ind+2} =  \ldots =
                x_1^e\tau_{e} = 0 \mod (x_1^e, x_2^e).
            \end{align}
            Together,
            Equations~\eqref{eq:specialindex},~\eqref{eq:godownindexes}~\eqref{eq:goupindexes}
            imply that $0 = \lambda_{\ind}x_1^{e-1}x_{2}^{e-1}\mod (x_1^e,
            x_2^e)$,
            so $\lambda_{\ind} = 0$. Since $\ind$ is arbitrary, we have
            $\lambda_i = 0$ for all $i$. As a result, the map $\varphi$ lifts to $\Hom(\IM, \OM)$ and the claim
            follows.
        \end{proof}

        \begin{proof}[Proof of Theorem~\ref{ref:maintheoremexamples:thm}]
            By Remark~\ref{ex:smallextensions},
            Corollary~\ref{ref:Msmoothexample:cor}, and
            Proposition~\ref{ref:righttarrowexample:prop} the
            assumptions of Theorem~\ref{ref:smoothness:thm} are satisfied for
            $\DM := \DM(e)$ and $\DR := \DR(e)$.
            A chase on Diagram~\eqref{eq:diagram}, taking into account
            surjectivity of $\uppleftarrow$ and $\righttarrow$, shows that
            \begin{align*}
                \Hom(\IR, \OR)_{<0} &= \ker\left( \Hom(\IM, \OR)_{<0} \to
                \Ext^1(\Ires, \OR)_{<0} \right)\\ &= \ker\left( \Hom(\IM,
                \OM)_{<0} \to
                \Ext^1(\Ires, \OR)_{<0} \right) = \Ext^1(\Ires, \Ires)_{<0}
                \simeq \Ext^1(\kk, \kk)_{<0}.
            \end{align*}
            As a result, we obtain $\dim_{\kk}\Hom(\IR, \OR)_{<0} = n$. The
            ideal $\IR$ is homogeneous, so by Lemma~\ref{ref:independence:cor}
            the subscheme $\DR(e)$ has trivial
            negative tangents.
            Corollary~\ref{ref:smoothnessHilb:cor} implies that
            $\compo(e)$ is elementary.
            Formula~\eqref{eq:dimensioncomparisonsnd} yields
            \[\dim \compo(e) = 4 + (4\deg \DM - 2e(e+1)) - 0 + (e^2 - 1)  -
                (2e+2) = 4\deg \DR
                - (e-1)(e+5).\qedhere
            \]
        \end{proof}

        \begin{remark}\label{ref:rationalprops:rmk}
            \def\Zflag{\compo_{\mathrm{flag}}}%
            \def\Zflagzero{\compo_{\mathrm{flag}}^{+}}%
            In the setting of Theorem~\ref{ref:maintheoremexamples:thm},
            denote by $\Zflag$ the component of $\HilbFlagplus$ containing $p =
            [\DR \subset \DM]$.
            By Remark~\ref{ex:smallextensions}, we have $\Ext^1(\Ires,
            \OR)_{\geq 0} = 0$ so $\Hom(\IR, \OR)_{\geq 0} \to \Hom(\IM,
            \OR)_{\geq 0}$ is
            surjective, so the tangent map of $\pi_{\DM}\colon\Zflag \to \im
            \pi_{\DM} \subset \Hilbmplus$ is
            surjective at the point $p$.
            Counting dimensions, we see that the general fiber of $\pi_{\DM}$
            is $(e^2-1)$-dimensional, so $\Zflag$ is dominated by a family
            of
            $\mathbb{P}^{e^2-1}$ and also $\compo(e)$ is dominated by such a
            family.
        \end{remark}

        Our proof of Theorem~\ref{ref:maintheoremexamples:thm} does not give
        an explicit description of the components $\compo(e)$. Below we
        describe $\compo(2)$ and $\compo(3)$, with
        their reduced scheme structure.
        \begin{example}\label{ex:exampletwo}
            The component $\compo(2)$ was discovered
            in~\cite[Section~2.2]{emsalem_iarrobino_small_tangent_space}.
            It is isomorphic to $\Gr(3, 10) \times \mathbb{A}^4$.
            This component was throughly analysed in~\cite{CEVV}.
        \end{example}

        \begin{example}\label{ex:examplethree}
            \def\hvec{\mathbf{h}}%
            \def\familygraded{\mathcal{F}}%
            \def\Vnot{V_0}%
            \def\multigraded{\OHilb^{\hvec}(\mathbb{A}^4)}%
            \def\compozero{\mathcal{Z}^{\Gmult}}%
            \def\formsres{\mathcal{F}^{\mathrm{res}}}%
            \def\forms{\mathcal{F}}%
            In contrast with $\compo(2)$, the component $\compo(3)$ was not known
            before. It is more complicated than $\compo(2)$ and we do not know
            if it is rational.

            The Hilbert function of $\OO_{\DR(3)}$ is $\hvec = (1, 4, 10, 12,
            8)$ and
            the Hilbert series of $\Hom(I_{\DR(3)}, \OO_{\DR(3)})$ is
            $4T^{-1} + 56 + 64T$ so the component $\compo$ is a rank $68$ fiber bundle
            over $\compozero$ and $\dim \compozero = 56$.
            Let $\multigraded$ be the multi-graded Hilbert scheme~\cite{Haiman_Sturmfels__multigraded}, parameterizing graded
            subschemes with Hilbert function $\hvec$.
            The scheme $\multigraded$ is naturally identified with
            \[
                \formsres := \left\{ (I_3, I_4) \in \Gr(8, \DS_3) \times \Gr(27, \DS_4)\ | \
                \DS_1 \cdot I_3 \subset I_4\right\}.
            \]
            Projection to first coordinate maps $\formsres$ onto a
            determinantal scheme
            \[
                \forms := \left\{I_3\in \Gr(8, \DS_3)\ |\ \dim_{\kk}(I_3\cdot\DS_1) \leq 27
                \right\} \subset \Gr(8, \DS_3),
            \]
            which is given by $28 \times 28$ minors, thus
            its dimension at every point is at
            least $8\cdot 12 - (35-27)\cdot (4\cdot 8 - 27) = 56$.  By
            comparing bounds, we see that equality occurs near $[\DR(3)]$, $56
            = \dim_{[\DR(3)]} \formsres = \dim_{[\DR(3)]}\multigraded$ and so
            $\compozero$ is equal to the unique component of $\formsres$ passing
            through $[\DR(3)]$.
            Since $\forms$ is determinantal, it has an embedding into the
            generic determinantal variety $D$.
            The map $\formsres \to \forms$ is the pullback of the resolution of
            singularities of $D$ via the embedding $\forms \subset D$, see~\cite[Section
            II.2]{Arbarello__curves_vol1}.
        \end{example}

        \begin{example}\label{ex:groebnerfans}
            In this example we construct a one-parameter family of finite subschemes
            $\DR_t$ of $\mathbb{A}^4$. The schemes $\DR_t$ have degree $35$
            and satisfy the following conditions
            \begin{enumerate}
                \item the Gr\"obner fans of all schemes $\DR_t$ are equal,
                \item for $t\neq 0$ the point $[\DR_t]\in
                    \Hlong{35}{\mathbb{A}^4}$ is smooth lying on an elementary component $\compo(3)$,
                \item the point $[\DR_0]$ lies in the intersection of two
                    components of $\Hlong{35}{\mathbb{A}^4}$ and it is
                    cleavable. In other words, $\DR_0$ is a limit of reducible
                    subschemes.
            \end{enumerate}
            Let $\kk = \mathbb{F}_3$. Put $s_t = {x}_{1}^{2}
            {x}_{3}^{2}-{x}_{2}^{2} {x}_{3}^{2}+t{x}_{1} {x}_{2} {x}_{3}
            {x}_{4}-{x}_{1}^{2} {x}_{4}^{2}-{x}_{2}^{2} {x}_{4}^{2}$ and $\DR_t
            = (x_1, x_2)^3 + (x_3, x_4)^3 + s_t$. The assertions on $[\DR_t]$
            for $t\neq 0$ follow immediately from
            Theorem~\ref{ref:maintheoremexamples:thm} as the form $s_t$
            corresponds to a matrix
            \[
                \begin{bmatrix}
                    1 & 0 & -1\\
                    0 & t & 0\\
                    -1 & 0 & -1\\
                \end{bmatrix},
            \]
            which is invertible for $t\neq 0$. Also the equality of Gr\"obner
            fans is clear. The point $[\DR_0]$ lies
            on $\compo(3)$. The family
            \[
                \Spec \frac{\kk[x_1, x_2, x_3, x_4][z]}{(x_3, x_4)^3 + (x_1x_2^2,
                  x_1x_2\cdot z-x_1^2x_2, x_1^2\cdot z-x_1^3,
                    (x_1x_3^2-x_1x_4^2)\cdot z-s_0)}\to \Spec \kk[z],
            \]
        is finite and flat. Its fiber over $z = 0$ is $\DR_0$, and its fiber over $z = 1$
        is reducible and so the image of $\Spec \kk[z]\to
        \Hlong{35}{\mathbb{A}^4}$ is contained in an irreducible component
        other that $\compo(3)$ and $[\DR_0]$
        lies in the intersection of those two components.
        \end{example}

        \begin{example}\label{ex:naiveexample}
            Suppose we know the Betti table of a finite subscheme which has trivial
            negative tangents.
            In this example we show how may sometimes deduce that this
            subscheme is a smooth point of the Hilbert scheme.
            Let $\DS = \kk[x,y,z,t]$.
            Let $\DR \subset \mathbb{A}^4 = \Spec \DS$ be given by $\IR = (x^3, y^3, z^3,
            t^3, Q_1, Q_2)$, where $Q_1$, $Q_2$ are quartics.
            For $\kk = \mathbb{F}_2$ and $Q_1 = x y^{2} z+y^{2}
                  z^{2}+x^{2} y t$, $Q_2 = y z^{2} t+x z t^{2}$
                  we have $H_{\OO_R}(T) = 1+4 T+10 T^{2}+16 T^{3}+17 T^{4}+8
                  T^{5}$,
            $\deg R = 56$ and the
            following Hilbert
            series
            \begin{equation}\label{eq:extsexplicit}
                H_{\Hom(\IR, \OR)}(T) = 4 T^{-1}+98+84 T+32 T^{2}.
            \end{equation}
            The resolution of $\OR$ is
            \begin{equation}\label{eq:resolution}
                \begin{tikzcd}
                    \DS & \arrow[l] \DS(-3)^{4} \oplus
                    \DS(-4)^{2} & \arrow[l] \DS(-6)^{16} \oplus
                    \DS(-7)^{4} & \arrow[l]  \ldots
                \end{tikzcd}
            \end{equation}
            Comparing~\eqref{eq:resolution} with $H_{\OO_R}$, we see that
            $\Ext^{1}(\IR, \OR)_{\geq 0} = 0$ by degree reasons.
            Since $\Ext^1$ has no non-negative part and $R$ has only trivial
            negative tangents,
            Corollary~\ref{ref:elementary:cor} implies that $[R]$ is a
            smooth point of an elementary component $\compo$ of dimension $\dim
            \Hom(\IR, \OR) = 218$. Since $\deg R = 56$ and $56\cdot 4 > 218$,
            the component $\compo$ is small and elementary.
            \emph{Macaulay2}~\cite{M2} experiments
            suggest that~\eqref{eq:extsexplicit} is true for a general choice of
            $Q_1$, $Q_2$.
            We stress that we deduced $\Ext^1_{\geq 0} = 0$
            only from the graded Betti numbers of $\OR$.
            The group $\Ext^1$
            is non-zero in negative degrees, with Hilbert function
            \[
                H_{\Ext^1(\IR, \OR)}(T) = 60 T^{-3}+204 T^{-2}+60 T^{-1}.
            \]
        \end{example}

        \begin{remark}\label{ref:previouswork:rmk}
            Below we discuss examples of elementary components known before this
            work. To the author's best knowledge, these are all the examples found
            in the literature. Trivially, $\Hlong{1}{\mathbb{A}^n} =
            \mathbb{A}^n$ is elementary.
            The first two nontrivial examples, with Hilbert function $(1, 4, 3)$ and $(1, 6, 6, 1)$
            respectively, are given
            in~\cite{emsalem_iarrobino_small_tangent_space}.
            For any integers $(d, e)$ with
            \[
                3\leq e \leq \frac{(d-1)(d-2)}{6}+2
            \]
            an elementary component with Hilbert function $(1, d, e)$
            is given
            in~\cite{Shafarevich_Deformations_of_1de}. Five examples,
            with Hilbert function $(1, n, n, 1)$, $8\leq n\leq 12$, are given
            in~\cite[Lemma~6.21]{iarrobino_kanev_book_Gorenstein_algebras}.
            Two other examples are given
            in~\cite[Corollaries~6.28,
            6.29]{iarrobino_kanev_book_Gorenstein_algebras}.
            Five
            examples are discovered in~\cite{HuibregtseElementary}.
            Only six
            of all these elementary components are small: those with Hilbert
            function $(1, 4, 3)$, $(1, 5, 3)$, $(1, 5, 4)$, $(1, 6, 6, 1)$ and
            two examples $(1, 5, 3, 4)$, $(1, 5, 3, 4, 5, 6)$
            from~\cite{HuibregtseElementary}. The elementarity of all these examples
            except the last two, which crucially employ non-graded schemes, can be proved using
            Corollary~\ref{ref:elementaryintro:cor} and the obtained
            components are products of $\mathbb{A}^n$ with Grassmannians.
            Therefore, they are smooth and rational.

            Both
            Iarrobino-Kanev~\cite[Conjecture~6.30]{iarrobino_kanev_book_Gorenstein_algebras},
            relying on~\cite{emsalem_iarrobino_small_tangent_space},
            and Huibregtse~\cite[Conjecture~1.4]{HuibregtseElementary}
            state conjectures which would give other infinite families of
            elementary components. To prove these conjectures it is, in each
            case, enough to check that the tangent space to the Hilbert scheme
            at a given point has expected dimension.
        \end{remark}

        \vspace*{-\baselineskip}
        \small
\newcommand{\etalchar}[1]{$^{#1}$}

    \end{document}